\documentclass[10pt]{article}

\usepackage[colorlinks,linkcolor=blue,citecolor=blue]{hyperref}
\usepackage{amsthm}
\usepackage{multirow}
\usepackage{marginnote}
\usepackage{a4wide}
\usepackage{amssymb}
\usepackage{amsfonts}
\usepackage{amsmath}
\usepackage{mathrsfs}
\usepackage{tikz}
\usetikzlibrary{arrows,matrix}
\usetikzlibrary{positioning}
\usepackage{mdframed} 
\usepackage{lipsum} 
\usepackage{extarrows} 
\usepackage{color}
\usepackage{enumitem}
\usepackage{tocloft} 

\input xy
\xyoption{arrow} \xyoption{matrix}

\date{}

\newtheorem{proposition}{Proposition}[section]
\newtheorem{theorem}[proposition]{Theorem}
\newtheorem{lemma}[proposition]{Lemma}

\newtheorem{corollary}[proposition]{Corollary}

\def\GK{{\rm  GK}\,}

\def\der{\partial }

\def\nFM0{{\nu }_{F,M_0}}
\def\nFN0{{\nu }_{F,N_0}}
\def\nGN0{{\nu }_{G,N_0}}

\def\N0{ {\bf N}_0 }

\def\Xpm{X^{\pm }}

\def\s{\sigma}
\def\Z{\mathbb{Z}}

\def\l1{{\lambda}_1}

\def\a{\alpha}
\def\a0{ {\alpha }_0}
\def\a1{ {\alpha }_1}

\def\l{\lambda}
\def\o{\omega}

\def\nFGM0{{\nu }_{F,G,M_0}}

\def\nFN0{{\nu}_{F,N_0}}

\def\sm{{\sigma}^m}

\def\sm1{{\sigma}^{-1}}

\def\smtp1{{\sigma}^{-t+1}}

\def\o{\omega }
\def\S1{S^{-1}}

\def\Xpm1{X^{\pm 1}_1}

\def\sPM1{{\sigma }^{\pm 1}}
\def\sMP1{{\sigma }^{\mp 1 }}

\def\d{\delta}

\def\di{{\rm d.ind}}

\def\L{\Lambda}

\def\CA{{\cal A}}

\def\Ytm1{Y^{t-1}}
\def\Yim1{Y^{i-1}}

\def\CN{{\cal N}}
\def\CS{{\cal S}}

\def\Aut{{\rm Aut}}

\def\dim{{\rm dim }}

\def\ker{ {\rm ker } }

\def\gr{ {\rm gr} }

\def\D{ \Delta }

\def\SL2Z{ {\rm SL}_2({\bf Z}) }

\def\CR{ {\cal R}}

\def\Gp1{ G^{1 , 1 } }
\def\P11{ P^{-1 , 1 } }
\def\Pp1{ P^{1 , 1 } }

\def\Supp{{\rm Supp}}

\def\nCLsr{{}^\nu\kern-2pt {\cal L}^{\sigma , \rho  }}
\def\nP{{}^\nu \kern-2pt P}
\def\nL{{}^\nu\kern-2pt L}
\def\nLL{{}^\nu\kern-2pt \Lambda}
\def\nPsr{{}^\nu\kern-2pt P^{\sigma , \rho  }}
\def\nLsr{{}^\nu\kern-2pt L^{\sigma , \rho  }}
\def\nuCL{{}^\nu\kern-2pt  {\cal L}}
\def\nCLsr{{}^\nu\kern-2pt {\cal L}^{\sigma , \rho  }}
\def\nCL1m{{}^\nu\kern-2pt {\cal L}^{-1 , 1  }}
\def\x1nu{x^\frac{1}{\nu}}
\def\xm1nu{x^{-\frac{1}{\nu}}}

\def\CR{ {\cal R}}

\def\CN{{\cal N}}

\def\CB{{\cal B}}

\def\CI{{\cal I}}

\def\CP{ {\cal P}}

\def\nAM0{{\nu }_{{\cal A},M_0}}
\def\nAN0{{\nu }_{{\cal A},N_0}}

\def\CR{ {\cal R }}
\def\CP{ {\cal P }}

\def\Gg{\mathfrak{g}}
\def\ga{\mathfrak{a}}
\def\gb{\mathfrak{b}}

\def\gm{\mathfrak{m}}
\def\gp{\mathfrak{p}}
\def\gq{\mathfrak{q}}
\def\gr{\mathfrak{r}}

\def\SL{{\rm SL}}

\def\Spec{{\rm Spec}}

\def\di!{\frac{\der^i}{i!}}
\def\dik!{\frac{\der^k_i}{k!}}

\def\Max{{\rm Max}}

\def\N{\mathbb{N}}

\def\0{\overline{0}}
\def\1{\overline{1}}

\def\Ln1{\L_{n,\overline{1}}}

\def\a1{a_{\overline{1}}}

\def\S{\Sigma}

\def\vn1{\overrightarrow{n-1}}
\def\CQ{{\cal Q}}

\def\Gr{{\rm gr}}

\def\sl{{\rm sl}}

\def\soc{{\rm soc}}

\def\mJ{\mathbb{J}}
\def\mI{\mathbb{I}}

\def\ann{{\rm ann}}

\def\K1{{\rm K}_1}

\def\mK{\mathbb{K}}

\def\Supp{{\rm Supp}}

\def\hmI1{\widehat{\mI_1}}
\def\tmI1{\widetilde{\mI_1}}
\def\tmJ1{\widetilde{\mJ_1}}
\def\hB1{\widehat{B_1}}
\def\hCB1{\widehat{\CB_1}}

\def\ga{\mathfrak{a}}

\def\tor{{\rm tor}}

\def \S{\mathcal{S}}

\def\sl2{{\mathfrak{sl}}_2}

\def\mK{\mathbb{K}}
\def\Prim{{\rm Prim}}

\makeatletter
\newenvironment{proof*}[1][\proofname]{\par
  \pushQED{\qed}%
   \vspace{-\topsep}
    \pushQED{\qed}%
    \normalfont
    \topsep0pt \partopsep0pt 
  \normalfont \partopsep=\z@skip \topsep=\z@skip
  \trivlist
  \item[\hskip\labelsep
        \itshape
    #1\@addpunct{.}]\ignorespaces
}{%
  \popQED\endtrivlist\@endpefalse
}
\makeatother

\begin{document}

\author{V. V. \  Bavula and T. Lu 
 }
\title{The prime spectrum and simple modules over the quantum spatial ageing algebra}

\maketitle

\begin{abstract}
For the algebra $\CA$ in the title, its prime, primitive and maximal spectra are classified. The group of automorphisms of $\CA$ is determined. The simple unfaithful $\CA$-modules and the simple weight $\CA$-modules are classified.

$\noindent $

{\em Key Words:  Prime ideal, maximal ideal, primitive ideal,  simple module, torsion module,  algebra automorphism, quantum algebra.}

{\em Mathematics subject classification 2010: 17B10, 16D25, 16D60, 16D70, 16P50, 17B37, 17B40.}

\small{\tableofcontents}
\end{abstract}


\section{Introduction} \label{Introduction} 

Let $\mathbb{K}$ be a filed and an element $q \in \mK^*:=\mK\setminus \{0 \}$ which is not a root of unity. The algebra $\mK_q[X, Y]:= \mK \langle X, Y \,|\, XY=qYX \rangle$ is called the \textit{quantum plane}. A classification of simple modules over the quantum plane is given in \cite{Bav-SimpModQuanPlane}.  The \textit{quantized enveloping algebra} $U_q(\sl2)$ of $\sl2$ is generated over $\mK$ by elements $E, F, K$ and $K^{-1}$ subject to the defining relations:
\begin{align*}
KEK^{-1}=q^2E, \,\,  KFK^{-1}=q^{-2}F,\,\, EF-FE = \frac{K-K^{-1}}{q- q^{-1}}.
\end{align*}
For basic properties and representation theory of the algebra $U_q(\sl2)$ the reader is referred to \cite{Jantzen-QuantGroup,Kassel-QuamtumGp}. The simple $U_q(\sl2)$-modules were classified in \cite{Bav-SimGWA-1992}, see also \cite{Bav-THM-GLGWA-1996}, \cite{Bav-SimpModQuanPlane} and \cite{Bav-OystGeneCrossPro}.  The quantum plane and the quantized enveloping algebra $U_q(\sl2)$ are important examples of generalized Weyl algebras and ambiskew polynomial rings, see e.g., \cite{Bav-GWArep} and \cite{D.A.Jordan-SimpleAmbiskew}. Let $U_q^{\geqslant 0}(\sl2)$ be the `positive part' of $U_q(\sl2)$. It is the subalgebra of  $U_q(\sl2)$ generated by $K^{\pm 1}$ and $E$. There is a Hopf algebra structure on $U^{\geqslant 0}_q(\sl2)$ defined by
\begin{align*}
\D(K)&= K \otimes K,   & \varepsilon(K) &=1, & S(K)&= K^{-1}, \\
\D(E)&= E \otimes 1+ K \otimes E, & \varepsilon(E)&=0, & S(E) &= -K^{-1}E.
\end{align*}
The notion of smash product has proved to be very useful in studying Hopf algebra actions \cite{Montgomery}. For example, the enveloping algebra of a semi-direct product of Lie algebras can naturally be seen as a smash product algebra. The smash product is constructed from a module algebra, see \cite[4.1]{Montgomery} for details and examples. We can make the quantum plane a $U^{\geqslant 0}_q(\sl2)$-\emph{module algebra} by defining
\begin{align*}
K \cdot X = qX,\,\,  E \cdot X= 0, \,\, K \cdot Y =q^{-1}Y,  \,\, E\cdot Y =X,
\end{align*}
and introduce the smash product algebra $\CA:=\mK_q[X, Y]\rtimes U_q^{\geqslant 0}(\sl2)$. We call this algebra the \emph{quantum spatial ageing algebra}.  The defining relations for the algebra $\CA$ are given explicitly in the following definition.

\noindent \textbf{Definition.} The \emph{quantum spatial ageing algebra} $\CA= \mK_q[X, Y]\rtimes U_q^{\geqslant 0}(\sl2)$ is an algebra generated over $\mathbb{K}$ by the elements $E,\,  K,\,  K^{-1},\, X$ and $ Y$ subject to the defining relations:
\begin{align}
EK &= q^{-2} KE,  & XK &=q^{-1} KX,  \,& YK &= qKY, \label{Rel1S} 
\\
EX &= qXE,\, & EY &= X+ q^{-1}YE, & qYX &=XY. \label{Rel2S} 
\end{align}
 The algebra $\CA$ can be seen as the quantum analogue of the enveloping algebra $U(\ga)$ of the 4-dimensional (non-semisimple) Lie algebra $\ga$ with basis $\{ h, e, x, y \}$ and Lie brackets:
$$[h, e]=2e, \,\, [h, x]=x, \,\, [h, y]=-y, \,\,  [e, x]=0,  \,\, [e, y]=x,\,\, [x, y]=0.$$
The Lie algebra $\ga$ is called the \emph{1-spatial ageing algebra}, it is studied in \cite{LMZ-AgingAlg} where all the simple weight modules are classified. Our aim is to study the structure of the algebra $\CA$ and its representation theory. For an infinite dimensional non-commutative algebra, to classify its simple modules is a very difficult problem, in general. The known examples of algebras for which the simple modules were classified are mainly generalized Weyl algebras, see e.g. \cite{Bav-SimGWA-1992,Bav-GWArep,Bav-THM-GLGWA-1996,Bav-SimpModQuanPlane,Bav-OystGeneCrossPro,Bav-OystWit-Woron,Bav-SimpHolonWeylA2,Block-IrrRepsl2} and Ore extensions with Dedekind ring as coefficient ring, \cite{Bav-OystWit-Woron}. In the paper, we classify all simple unfaithful $\CA$-modules and the simple weight modules. Weight modules for certain (quantum) algebras of small Gelfand-Kirillov dimension are treated in \cite{Dobrev-LowestweightSchrodg,Dobrev-qSchrodinger,Dubsky-WeightSchrodg,Dubsky-Lv-Mazorchuk-Zhao-CatOSchrodinger,Wu-IndecompModSchrodg,Wu-Zhu-SimpweighModSchrodg}.

The paper is organized as follows. In Section \ref{PrimeSpec}, we describe the partially ordered sets of the prime, maximal and primitive ideals of the algebra $\CA$. Using this description the prime factor algebras of $\CA$ are given explicitly via generators and relations (Theorem \ref{A23Feb15}). There are nine types of prime factor algebras of $\CA$. For two of them, `additional' non-obvious units appear under factorization at prime ideals. It is proved that every prime ideal of $\CA$ is completely prime (Corollary \ref{a13Mar15}). In Section \ref{AutoGroup}, the automorphism group of $\CA$ is determined, which turns out to be a `small' non-commutative group that contains an infinite discrete subgroup (Theorem \ref{B27Feb15}). The orbits of the prime spectrum under the action of the automorphism group are described. In Section \ref{ClaANNz}, we classify all simple unfaithful $\CA$-modules. This is the first instance where such a large class of simple modules is classified for a quantum algebra of Gelfand-Kirillov dimension larger than 2. In Section \ref{SimWeight}, we find the centralizer $C_{\CA}(K)$ of the element $K$ in the algebra $\CA$ and give a classification of all simple weight $\CA$-modules.

The prime or/and primitive ideals of various quantum algebras (and their classification) are considered in \cite{Brown-Goodearl-LectQutumGp,Cauchon-SpecOqMn,Dumas-Rigal-SpecAutJordanMatrix,Good-Letzter-PrimFactQuanMatr,Good-Letzter-PrimIdealqSkew,Good-Lenagan-CatenarityQuanAlg,Good-Letzter-PrimMulQaffin,Launois-Uq(B2),Lopes-PrimUqsln,Malliavin-QuanHeisenberg,Rigal--SpecWeyl}. The automorphism group of some (quantum) algebras are considered in \cite{Alev-Chamarie-DerAut-92,Andrumskiewitsch-Duams-AutUqg,Bav-Jordan-IsoGWA,Fleury-AutUqb,Joseph-WildAutUsl2,Launois-Lopes-AutUqsl4,Launois-Lenagan-PrimQuanMatr,Shestakov-Umirbaev-AutPolythree}.
\section{Prime spectrum of the algebra $\CA$} \label{PrimeSpec} 
The aim of this section is to describe the prime, maximal and primitive spectrum of the algebra $\CA$ (Theorem \ref{A23Feb15}, Corollary \ref{a24Feb15} and Proposition \ref{a27Feb15}). Every prime ideal of $\CA$ is completely prime (Corollary \ref{a13Mar15}). For all prime ideals $P$ of $\CA$, the factor algebras $\CA/P$ are given by generators and defining relations (Theorem \ref{A23Feb15}).

{\bf Generalized Weyl algebra}. {\it Definition}, \cite{Bav-GWA-FA-91, Bav-UkrMathJ-92, Bav-GWArep}. Let $D$ be a ring, $\s$ be an automorphism of $D$ and $a$ is an element of the centre of $D$. {\em The generalized Weyl algebra}
$A:=D(\sigma, a):=D[X,Y; \s , a]$ is a ring  generated by $D$,
$X$ and $Y$ subject to the defining relations:
$$
X\alpha=\sigma(\alpha)X \;\; {\rm and}\;\;  Y\alpha=\sigma^{-1}(\alpha)Y\;\;  {\rm for \; all}\;\;
\alpha \in D, \;\;  \ YX=a \;\;  {\rm and}\;\; XY=\sigma(a).
$$
The algebra $A={\oplus}_{n\in {\bf \Z}}\, A_n$
is $\Z$-graded where $A_n=Dv_n$,
$v_n=X^n$ for $n>0$, $v_n=Y^{-n}$ for $n<0$ and $v_0=1.$
 It follows from the above relations that
$v_nv_m=(n,m)v_{n+m}=v_{n+m}\langle n,m\rangle $
for some $(n,m)$ and $\langle n, m \rangle \in D$. If $n>0$ and $m>0$ then
\begin{eqnarray*}
 n\geq m:& &  (n,-m)=\sigma^n(a)\cdots \sigma^{n-m+1}(a),\;\;  (-n,m)=\sigma^{-n+1}(a)
\cdots \sigma^{-n+m}(a),\\
n\leq m: & & (n,-m)=\sigma^{n}(a)\cdots \sigma(a),\,\,\,\;\;\; \;\;\; \;\;\; (-n,m)=\sigma^{-n+1}(a)\cdots a,
\end{eqnarray*}
in other cases $(n,m)=1$. Clearly, $\langle n,m\rangle =\s^{-n-m}((n,m))$.

{\it Definition}, \cite{Bav-GlGWA-1996}. Let $D$ be an ring and $\s$ be its automorphism.  Suppose that elements $b$ and $\rho $
belong to the centre of the ring $D$,  $\rho $ is invertible and
 $\s (\rho )=\rho$.  Then   $E:=D \langle \s ;b,\rho\rangle:= D[ X, Y; \s ,b,\rho ]$ is a ring generated by  $D$, $X$ and $Y$ subject to the defining
relations:
$$X\alpha =\s (\alpha )X\;\; {\rm and} \;\; Y\alpha =\s {}^{-1}(\alpha )Y \;\; {\rm for \; all}\;\;   \alpha \in D, \;\; {\rm and} \;\;\, XY-\rho
YX=b.$$

If $D$ is commutative domain, $\rho =1$ and $b = u-\s (u)$ for some $u\in D$ (resp., if $D$ is a commutative finitely generated domain over a field $\mK$ and  $\rho \in \mK^*$) the algebras $E$ were considered in \cite{Jordan-ItSkewPol-1993} (resp., \cite{Jordan-Wells-1996}).

The ring $E$ is {\em the iterated skew polynomial ring}
$E=D[Y;\sm1  ][X;\s ,\der ]$ where $\der $ is the $\s$-derivation of
$D[Y;\sm1 ]$ such that $\der D=0$ and $\der Y=b$ (here the automorphism $\s $ is extended
from $D$ to $D[Y;\sm1 ]$ by the rule $\s (Y)=\rho Y)$.

An element $d$ of a ring $D$ is {\em normal} if $dD=Dd$.
 The next proposition shows that the rings $E$ are GWAs and under a certain (mild) conditions they have a `canonical' normal element.

\begin{proposition}\label{A28Mar15}
Let $E=D[X, Y;\s , b, \rho ]$. Then
\begin{enumerate}
\item \cite[Lemma 1.2]{Bav-GlGWA-1996}   The ring $E$ is the GWA $D[H][X,Y; \s , H]$ where $\s (H) = \rho H+b$.
    \item \cite[Lemma 1.3]{Bav-GlGWA-1996}  The following statements are equivalent:
\begin{enumerate}
\item \cite[Corollary 1.4]{Bav-GlGWA-1996} $C=\rho(YX+\alpha )=XY+\s (\alpha )$ is a normal element  in $E$ for some central  element $\alpha \in D$,
\item $\rho\alpha -\s (\alpha )=b$ for some central element $\alpha \in D$.
\end{enumerate}
\item \cite[Corollary 1.4]{Bav-GlGWA-1996}    If one of the equivalent conditions of statement 2 holds then the ring $E=D[C][X,Y; \s, a=\rho^{-1}C-\alpha ]$ is a GWA where $\s (C)=\rho  C$.
\end{enumerate}
\end{proposition}

\textbf{The algebra $\mathbb{E}$ is a GWA.} Let $\mathbb{E}$ be the subalgebra of $\CA$ generated by the elements $X, E$ and $Y$. The generators of the algebra $\mathbb{E}$ satisfy the defining relations
\begin{align*}
EX= qXE, \quad YX = q^{-1}XY \quad {\rm and}\,\, EY- q^{-1}YE = X.
\end{align*}
So, $\mathbb{E}= \mK[X] [ E, Y; \sigma, b= X, \rho = q^{-1} ]$ where $\sigma(X)= qX$. The polynomial $\alpha = \frac{X}{q^{-1}-q}$ is a solution to the equation $q^{-1}\alpha - \sigma(\alpha)= X$. By Proposition \ref{A28Mar15}, the algebra $\mathbb{E}= \mK[X, C][E, Y; \sigma, a= qC- \alpha]$ is a GWA where $\sigma(X)=qX, \sigma(C)= q^{-1}C$ and $C= q^{-1}(YE + \frac{X}{q^{-1}-q})= EY + \frac{qX}{q^{-1}-q}$ is a normal element of the algebra $\mathbb{E}$. Then the element $\varphi = q(q^{-1}-q)C$ is a normal element of the algebra $\mathbb{E}$. Clearly, $\varphi = (q^{-1}-q)YE+X = (1-q^2)EY + q^2 X$. Then
\begin{align}
\mathbb{E}= \mK[X, \varphi][E, Y; \sigma, a = \frac{\varphi -X}{q^{-1}-q}] \label{EEX} 
\end{align}
is a GWA where $\sigma(X)=qX$ and $\sigma(\varphi)= q^{-1} \varphi$. So, the algebra
\begin{align}
\CA= \mathbb{E}[K^{\pm 1}; \tau] \label{EEX1} 
\end{align}
is a skew Laurent polynomial algebra where $\tau(E)= q^2 E, \tau(X)=qX, \tau(Y)=q^{-1}Y$ and $\tau(\varphi)= q \varphi$. The algebra $\CA$ is a Noetherian domain of Gelfand-Kirillov dimension $\GK(\CA)=4.$

Let $x_1, x_2, x_3$ be a permutation of the elements $X, Y, E$. Then $\{ K^i x^{\alpha}\,|\, i \in \Z, \alpha \in \N^3 \}$ is a $\mK$-basis for the algebra $\CA$ where $x^{\alpha}= x_1^{\alpha_1} x_2^{\alpha_2} x_3^{\alpha_3}$ and $\alpha = ( \alpha_1, \alpha_2, \alpha_3 )$. The standard filtration associated with the canonical generators of the algebra $\CA$ is equal to $\{ F_n \}_{n \geqslant 0}$ where $F_n = \sum_{|i|+|\alpha| \leqslant n} \mK K^i x^{\alpha}$ where $|\alpha| = \alpha_1 + \alpha_2 + \alpha_3.$ Therefore, the Gelfand-Kirillov dimension of the algebra $\CA$ is equal to $\GK(\CA)= 4.$

For a left denominator set $\CS$ of a ring $R$, we denote by $\CS^{-1}R= \{ s^{-1}r \,|\, s \in \CS, r \in R \}$ the left localization of the ring $R$ at $\CS$. If the left denominator set $\CS$ is generated by elements $X_1, \ldots, X_n$, we also use the notation $R_{X_1, \ldots, X_n}$ to denote the ring $\CS^{-1}R$. If $M$ is a left $R$-module then the localization $\CS^{-1}M$ is also denoted by $M_{X_1, \ldots, X_n}$.

\begin{lemma}\label{c23Feb15} 
The following identities hold in the algebra $\CA$.
\begin{enumerate}
\item $EY^i = \frac{q^{-2i}-1}{q^{-2}-1}XY^{i-1}+ q^{-i}Y^i E.$
\item $YE^i = q^i E^iY - \frac{q(1-q^{2i})}{1-q^2}XE^{i-1}.$
\end{enumerate}
\end{lemma}
\begin{proof}
Both equalities can be proved by induction on $i$ and using the relation $EY= X+ q^{-1}YE.$
\end{proof}

By Lemma \ref{c23Feb15}, the set $\CS_Y := \{ Y^i \,|\, i \geqslant 0 \}$ is a left and right Ore set in the algebra $\CA.$ Note that the algebra $U_q^{\geqslant 0}(\sl2)$ is the localization of its subalgebra, the quantum plane $\mK \langle K, E \,| KE=q^2 EK \rangle$, at the powers of the element $K$. Let $\CA_Y$ be the localization of $\CA$ at the powers of $Y$. Recall that  $\varphi= EY-qYE = X+(q^{-1}- q)YE$, we have
 \begin{align}
 X\varphi&=\varphi X,  & Y\varphi &= q\varphi Y, & E\varphi &= q^{-1}\varphi E, & K\varphi &= q\varphi K.&  \label{XqYE} 
\end{align}
So, the element $\varphi$ is a normal element of the algebra $\CA$. Recall that an element $x$ of an algebra $A$ is called a \emph{normal element} if $xA=Ax$. Then
\begin{align}
\CA_Y= \mK \langle K^{\pm 1}, E, Y^{\pm 1}, X \rangle = \mK \langle K^{\pm 1}, \varphi, Y^{\pm 1}, X\rangle = \mK[\varphi, X][Y^{\pm 1}; \sigma][K^{\pm 1}; \tau] \label{AYita} 
\end{align}
is an iterated skew polynomial ring where $\sigma$ is the automorphism of $\mK[\varphi, X]$ defined by $\sigma(\varphi)= q\varphi$, $\sigma(X)= q^{-1}X$; and $\tau$ is the automorphism of the algebra $\mK[\varphi, X][Y^{\pm 1};\sigma]$ defined by $\tau(\varphi)= q\varphi, \tau(X)= qX, \tau(Y)= q^{-1}Y$. Let $\CA_{Y, X, \varphi}$ be the localization of $\CA_Y$ at the denominator set $\{X^i\varphi^j \,|\, i, j \in \N \}$, then $\CA_{Y, X, \varphi}= \mK[\varphi^{\pm 1}, X^{\pm 1}][Y^{\pm 1}; \sigma][K^{\pm 1}; \tau]$ is a quantum torus.
For an algebra $A$ we denote by $Z(A)$ its centre. The next result shows that the algebra $\CA$ and some of its localizations have trivial centre.
\begin{lemma}\label{a22Feb15} 
\begin{enumerate}
\item $Z(\CA_{Y, X, \varphi})=\mK.$
\item $\CA_{Y, X, \varphi}$ is a simple algebra.
\item $Z(\CA_Y)=\mK.$
\item $Z(\CA)=\mK.$
\end{enumerate}
\end{lemma}
\begin{proof}
1. Let $u= \sum \alpha_{i,j,k,l} \varphi^i X^j Y^k K^l \in Z(\CA_{Y, X, \varphi})$, where $\alpha_{i,j,k,l} \in \mK$ and  $i,j, k, l \in \Z$. Since $Ku= uK$, we have $i+j-k=0$. The equality $Xu= uX$ implies that $k-l=0$. Similarly, the equality  $Yu= uY$ implies that $i-j+l =0$. Finally, using $\varphi u= u\varphi$ we get $-k-l=0$. Therefore, we have $i=j=k=l=0$, and so $u \in \mK$. Thus $Z(\CA_{Y, X, \varphi})=\mK.$

2. By \cite[Corollary 1.5.(a)]{Good-Letzter-PrimeQuanSpac}, contraction and extension provide mutually inverse isomorphisms between the lattices of ideals of a quantum torus and its centre. Then statement 2 follows from statement 1.

3. Since $\mK \subseteq Z(\CA_Y) \subseteq Z(\CA_{Y, X, \varphi})\cap \CA_Y =\mK$, we have $Z(\CA_Y)=\mK.$

4. Since $\mK \subseteq Z(\CA) \subseteq Z(\CA_Y) \cap \CA = \mK$, we have $Z(\CA)= \mK$.
\end{proof}

\begin{lemma} \label{b24May15} 
The algebra $\CA_{X, \varphi}$ is a central simple algebra.
\end{lemma}
\begin{proof}
By Lemma \ref{a22Feb15}.(1), the algebra $\CA_{Y, X, \varphi}$ is central, hence so is the algebra $\CA_{X, \varphi}$. By Lemma \ref{a22Feb15}.(2), the algebra $(\CA_{X, \varphi})_Y= \CA_{Y, X, \varphi}$ is a simple Noetherian domain. So, if $I$ is a nonzero ideal of the algebra $\CA_{X, \varphi}$ then $I_Y$ is a nonzero ideal of the algebra $\CA_{Y, X, \varphi}$, i.e., $I_Y= \CA_{Y, X, \varphi}$, and so $Y^i \in I$ for some $i \geqslant 0$. To finish the proof it suffices to show that
\begin{align}
(Y^i) = \CA_{X, \varphi} \quad {\rm for \,\, all\,\,} i \geqslant 1. \label{YiAA} 
\end{align}
To prove the equality we use induction on $i$. Let $i=1$. Then $X=EY - q^{-1}YE \in (Y)$. Since $X$ is a unit of the algebra $\CA_{X, \varphi}$, the equality (\ref{YiAA}) holds for $i =1$. Suppose that $i \geqslant 1$ and equality (\ref{YiAA}) holds for all $i'$ such that $i' < i$. By Lemma \ref{c23Feb15}.(1), $Y^{i-1} \in (Y^i)$. Hence, $(Y^i)= (Y^{i-1})= \CA_{X, \varphi}$, by induction.
\end{proof}

By (\ref{EEX}), the localization $\mathbb{E}_Y$ of the algebra $\mathbb{E}$ at the powers of the element $Y$ is the skew Laurent polynomial algebra $\mathbb{E}_Y = \mK[X, \varphi][Y^{\pm 1}; \sigma^{-1}]$ where $\sigma(X)=qX$ and $\sigma(\varphi)= q^{-1} \varphi$. Similarly, $\mathbb{E}_E = \mK[X, \varphi][E^{\pm 1}; \sigma] \simeq \mathbb{E}_Y$ where $\sigma(X)=qX$ and $\sigma(\varphi)= q^{-1} \varphi$. So, $\CA_Y= \mathbb{E}_Y[K^{\pm 1}; \tau]$. The element $\varphi$ is a normal element of the aglebras $\mathbb{E}$ and $\CA$. So, the localizations of the algebras $\mathbb{E}_Y$ and $\CA_Y$ at the powers of $\varphi$ are as follows
\begin{align}
\mathbb{E}_{Y, \varphi}= \mK[X, \varphi^{\pm 1}][Y^{\pm 1}; \sigma^{-1}] \,\,{\rm and}\,\, \CA_{Y, \varphi}= \mathbb{E}_{Y, \varphi}[K^{\pm 1}; \tau]. \label{EEX5} 
\end{align}

Now, we introduce several factor algebras and localizations of $\CA$ that play a key role in finding the prime spectrum of the algebra $\CA$ (Theorem \ref{A23Feb15}) and all the prime factor algebras of $\CA$ (Theorem \ref{A23Feb15}). In fact, explicit sets of generators and defining relations are found for all prime factor algebras of $\CA$ (Theorem \ref{A23Feb15}). Furthermore, all these algebras are domains, i.e., all prime ideals of  $\CA$ are completely prime (Corollary \ref{a13Mar15}). For an element $a$ of an algebra $A$, we denote by $(a)$ the ideal it generates.

\textbf{The algebra $\CA/(X)$.} The element $X$ is a normal element in the algebras $\mathbb{E}$ and $\CA$. By (\ref{EEX}), the factor algebra
\begin{align}
\mathbb{E}/(X) = \mK[\varphi][E,Y; \sigma, a= \frac{\varphi}{q^{-1}-q}], \quad \sigma(\varphi)= q^{-1}\varphi, \label{EEX2} 
\end{align}
is a GWA. Since $YE = \frac{\varphi}{q^{-1}-q}, EY = q^{-1} \frac{\varphi}{q^{-1}-q}= q^{-1}YE$, the algebra
\begin{align}
\mathbb{E}/(X) \simeq \mK \langle E, Y \,|\, EY = q^{-1}YE \rangle \label{EEX3} 
\end{align}
is isomorphic to the quantum plane. It is a Noetherian domain of Gelfand-Kirillov dimension 2. Now, the factor algebra
\begin{align}
\CA/(X) \simeq \mathbb{E}/(X)[K^{\pm 1}; \tau] \label{EEX4} 
\end{align}
is a skew Laurent polynomial algebra where $\tau(E)= q^2 E$ and $\tau(Y)= q^{-1}Y.$ It is a Noetherian domain of Gelfand-Kirillov dimension 3.
The element of the algebra $\CA/(X)$, $Z:= \varphi YK^{-1} = (1-q^2)EY^2K^{-1}$, belongs to the centre of the algebra $\CA/(X)$. By (\ref{EEX5}), the localization of the algebra $\CA/(X)$ at the powers of the central element $Z$,
\begin{align}
\Big(\CA/(X)\Big)_Z \simeq \frac{\CA_{Y,\varphi}}{(X)_{Y, \varphi}} \simeq \mK[Z^{\pm 1}] \otimes \mathbb{Y}, \label{EEX6} 
\end{align}
is the tensor product of algebras where the algebra $\mathbb{Y}:= \mK[Y^{\pm 1}][K^{\pm 1}; \tau]$ is a central simple algebra since $\tau(Y)=q^{-1}Y$ and $q$ is not a root of unity. Hence, the centre of the algebra $(\CA/(X))_Z$ is $\mK[Z^{\pm 1}]$. The algebra $(\CA/(X))_Z$ is a Noetherian domain of Gelfand-Kirillov dimension 3.

\begin{lemma} \label{a23May15} 
\begin{enumerate}
\item $Z\Big(\CA/(X)\Big)= \mK[Z]$.
\item $Z\bigg(\Big(\CA/(X)\Big)_Z\bigg)= \mK[Z^{\pm 1}]$.
\end{enumerate}
\end{lemma}
\begin{proof}
$Z\Big(\CA/(X)\Big) = \CA/(X) \cap Z\bigg(\Big(\CA/(X)\Big)_Z\bigg) = \CA/(X) \cap \mK[Z^{\pm 1}]= \mK[Z]$, by (\ref{EEX4}).
\end{proof}

\textbf{The algebra $\CA/(\varphi)$.} The element $\varphi$ is a normal element in the algebras $\mathbb{E}$ and $\CA$. By (\ref{EEX}), the factor algebra
\begin{align}
\mathbb{E}/(\varphi) = \mK[X][E, Y; \sigma, a = - \frac{X}{q^{-1}-q}], \quad \sigma(X)= qX,
\end{align}
is a GWA. Since $YE = - \frac{X}{q^{-1}-q}$ and $EY = q(- \frac{X}{q^{-1}-q})=qYE$, the algebra
\begin{align}
\mathbb{E}/(\varphi) \simeq \mK \langle E, Y \,|\, EY = qYE \rangle \label{EEP2} 
\end{align}
is isomorphic to the quantum plane. It is a Noetherian domain of Gelfand-Kirillov dimension 2. Now, the factor algebra
\begin{align}
\CA/(\varphi) \simeq \mathbb{E}/(\varphi) [K^{\pm 1}; \tau] \label{EEP3} 
\end{align}
is a skew Laurent polynomial algebra where $\tau(E)= q^2 E$ and $\tau(Y)= q^{-1}Y$. The algebra $\CA/(\varphi)$ is a Noetherian domain of Gelfand-Kirillov dimension 3. The element $C:= XYK \in \CA/(\varphi)$ belongs to the centre of the algebra $\CA/(\varphi)$.

The localization $\mathbb{E}_{X,Y}$ of the algebra $\mathbb{E}$ at the Ore set $\CS= \{ X^iY^j \,|\, i, j \in \N \}$,
\begin{align}
\mathbb{E}_{X,Y} = \mK[X^{\pm 1}, \varphi][Y^{\pm 1}; \sigma^{-1}], \;\; \sigma(X)=q X, \;\; \sigma(\varphi)= q^{-1} \varphi, \label{EEP4} 
\end{align}
is a skew Laurent polynomial algebra. Then the localization $\CA_{X,Y}$ of the algebra $\CA$ at the Ore set $\CS$ is equal to $\CA_{X,Y} = \mathbb{E}_{X,Y}[K^{\pm 1}; \tau]$. By (\ref{EEP2}) and (\ref{EEP3}), the localization of the algebra $\CA/(\varphi)$ at the powers of the element $C$,
\begin{align}
\bigg(\frac{\CA}{(\varphi)}\bigg)_C \simeq \frac{\CA_{X,Y}}{(\varphi)_{X,Y}} \simeq \bigg({\frac{\CA_X}{(\varphi)_X}}\bigg)_Y \simeq \mK[C^{\pm 1}] \otimes \mathbb{Y} \label{EEP5} 
\end{align}
is a tensor product of algebras where $\mathbb{Y}$ is the central simple algebra as in (\ref{EEX6}). Hence, the centre of the algebra $\Big(\CA/(\varphi)\Big)_C$ is $\mK[C^{\pm 1}]$.

\begin{lemma}\label{a24May15} 
\begin{enumerate}
\item $Z(\CA/\Big(\varphi)\Big)= \mK[C]$.
\item $Z\bigg(\Big(\CA/(\varphi)\Big)_C\bigg)= \mK[C^{\pm 1}]$.
\end{enumerate}
\end{lemma}
\begin{proof}
$Z(\CA/\Big(\varphi)\Big) = \CA/(\varphi) \cap Z\bigg(\Big(\CA/(\varphi)\Big)_C\bigg) = \CA/(\varphi) \cap \mK[C^{\pm 1}]= \mK[C]$, by (\ref{EEP3}).
\end{proof}

For an algebra $A$, $\Spec\,(A)$ is the set of its prime ideals. The set $(\Spec\,(A), \subseteq)$ is a partially ordered set (poset) with respect to inclusion. Let $f: A \rightarrow B$ be an algebra epimorphism. Then $\Spec\,(B)$ can be seen as a subset of $\Spec\,(A)$ via the injection $\Spec\,(B) \rightarrow \Spec\,(A), \,\, \gp \mapsto f^{-1}(\gp)$. So, $\Spec\,(B)= \{\gq \in \Spec\,(A) \,|\, \ker(f) \subseteq \gq \}$.
Given a left denominator set $\CS$ of the algebra $A$. Then $\sigma: A \rightarrow \CS^{-1}A,\,\, a \mapsto s^{-1}a$, is an algebra homomorphism. If the algebra $A$ is a Noetherian algebra then $\Spec\,(\CS^{-1}A)$ can be seen as a subset of $\Spec\,(A)$ via the injection $\Spec\,(\CS^{-1}A) \rightarrow \Spec\,(A), \,\, \gq \mapsto \sigma^{-1}(\gq)$.

Let $R$ be a ring. Then each element $r \in R$ determines two maps from $R$ to $R$,  $r \cdot: x \mapsto rx$ and $\cdot r: x \mapsto xr$ where $x \in R.$  Recall that for an element $r \in R$, we denote by $(r)$ the ideal of $R$ generated by the element $r$.

\begin{proposition}\label{aA12Mar15} 
Let $R$ be a Noetherian ring and $s$ be an element of $R$ such that $\CS_s := \{s^i \,|\, i\in \N \}$ is a left denominator set  of the ring $R$ and $(s^i)= (s)^i$ for all $i\geqslant 1$ (e.g., $s$ is a normal element such that $\ker(\cdot s_{R}) \subseteq \ker(s_R \cdot)$). Then
$\Spec\,(R)= \Spec(R, s) \, \sqcup \, \Spec_s(R) $ where $\Spec(R,s):= \{ \gp \in \Spec\,(R)\,|\, s \in \gp \}$,  $\Spec_s(R)= \{\gq \in \Spec\,(R)\, |\, s \notin \gq \}$ and
\begin{enumerate}[label=(\alph*)]
\item  the map $\Spec\,(R, s) \rightarrow \Spec\,(R/(s)),\,\, \gp \mapsto \gp/(s)$, is a bijection with the inverse $\gq \mapsto \pi^{-1}(\gq)$ where $\pi: R \rightarrow R/(s), r \mapsto r+(s),$
\item the map $\Spec_s(R) \rightarrow \Spec\,(R_s), \,\, \gp \mapsto \CS_s^{-1}\gp,$ is a bijection with the inverse $\gq \mapsto \sigma^{-1}(\gq)$ where $\sigma: R \rightarrow R_s:= \CS_s^{-1}R, \,r \mapsto \frac{r}{1}$.
\item For all $\gp \in \Spec\,(R, s)$ and $\gq \in \Spec_s(R), \gp \not\subseteq \gq.$
\end{enumerate}
\end{proposition}
\begin{proof}

Clearly, $\Spec\,(R)= A_1 \, \sqcup \, A_0$ is a disjoint union where $A_1$ and $A_0$ are the subsets of $\Spec\,(R)$ that consist of prime ideals $\gp$ of $R$ such that $\gp \cap \CS_s \neq \emptyset$ and $\gp \cap \CS_s = \emptyset$, respectively. If $\gp \in A_1$ then $s^i \in \gp$ for some $i \geqslant 1$, and so $\gp \supseteq (s^i)=(s)^i$, by the assumption. Therefore, $\gp \supseteq (s)$ (since $\gp$ is a prime ideal), i.e., $\gp \ni s$. This means that $A_1 = \Spec\,(R, s)$. We have shown that $s \in \gp$ iff $s^i \in \gp$ for some $i\geqslant 1$. By the very definition, $A_0 = \Spec\,(R) \setminus A_1= \Spec\,(R) \setminus \{ \gp \in \Spec\,(R)\,|\,s \in \gp \}= \{ \gp \in \Spec\,(A)\,|\, s \notin \gp \}= \Spec_s(R)$.

The statement (a) is obvious since $s \in \gp$ iff $(s) \subseteq \gp$. The ring $R$ is Noetherian, by \cite[Proposition 2.1.16.(vii)]{MR}, the map $\Spec_s(R)=\{\gp \in \Spec\,(R)\,|\, \gp \cap \CS_s = \emptyset \} \rightarrow \Spec\,(R_s), \gp \mapsto \CS_s^{-1}\gp$ is a bijection with the inverse $\gq \mapsto \sigma^{-1}(\gq)$ and the statement (b) follows. The statement (c) is obvious.
\end{proof}

In the proof of Theorem \ref{A23Feb15} the following very useful lemma is used repeatedly.
\begin{lemma}\label{a22May15} 
Let $A$ be a ring, $\CS$ be a left denominator set of $A$ and $\sigma: A \rightarrow \CS^{-1}A, \, a \mapsto \frac{a}{1}$. Let $\gq$ be a completely prime ideal of $\CS^{-1}A$, $\gp$ be an ideal of $A$ such that $\gp \subseteq \sigma^{-1}(\gq)$ and $\CS^{-1}\gp = \gq$. Then $\gp = \sigma^{-1}(\gq)$ iff $A/\gp$ is a domain.
\end{lemma}
\begin{proof}
($\Rightarrow$) Since $A/\sigma^{-1}(\gq) \subseteq \CS^{-1}A/\gq$ and $\CS^{-1}A/\gq$ is a domain (since $\gq$ is a completely prime ideal), the algebra $A/\sigma^{-1}(\gq)$ is a domain.

($\Leftarrow$) The left $\CS^{-1}A$-module $\CS^{-1}(A/\gp) \simeq \CS^{-1}A/\CS^{-1}\gp \simeq \CS^{-1}A/\gq$ is not equal to zero. In particular, $\CS \cap \gp = \emptyset$. So, for all $s \in \CS$, the elements $s+ \gp$ are nonzero in $A/\gp$. Since $A/\gp$ is a domain, $\tor_{\CS}(A/\gp)=0.$ Clearly, $\sigma^{-1}(\gq)/\gp \subseteq \tor_{\CS}(A/\gp).$ Hence, $\gp = \sigma^{-1}(\gq)$.
\end{proof}

\textbf{The prime spectrum of the algebra $\CA$.}
The key idea in finding the prime spectrum of the algebra $\CA$ is to use Proposition \ref{aA12Mar15} repeatedly and the following diagram of algebra homomorphisms
\begin{align}
\begin{tikzpicture}[semithick,scale=2.5, text centered]
\node (A) at (0, 2) {$\CA$};
\node (Ax) at (1, 2) {$\CA_X$};
\node (Axf) at (2, 2) {$\CA_{X, \varphi}$};
\node (Amx) at (0, 1.5) {$\CA/(X)$};
\node (Axmfx) at (1, 1.5) {$\CA_X/(\varphi)_X$};
\node (Amxz) at (0,1) {$\CA/(X, Z)$};
\node (Axlz) at (1, 1) {$\quad \Big(\CA/(X)\Big)_Z$};
\node (U) at (0, 0.5) {$U= \CA/(X, Z, Y)\quad \qquad$};
\node (Y) at (1, 0.5) {$\qquad\qquad\; \Big(\CA/(X, Z)\Big)_Y \simeq \mathbb{Y}$};
\node (L) at (0, 0) {$L= U/(E)$};
\node (UE) at (1, 0) {$U_E\quad$};
 \draw[->] (A) edge[left] (Ax);
 \draw[->] (Ax) edge[left] (Axf);
\draw[->] (A) edge[below] (Amx);
 \draw[->] (Amx) edge[below] (Amxz);
  \draw[->] (Amxz) edge[below] (U);
   \draw[->] (U) edge[below] (L);
    \draw[->] (Ax) edge[below] (Axmfx);
     \draw[->] (Amx) edge (Axlz);
      \draw[->] (Amxz) edge (Y);
       \draw[->] (U) edge (UE);
   \end{tikzpicture} \label{NAYXF} 
   \\
   \Big({\rm where\,\,} L=\mK[K^{\pm 1}] \,\,{\rm and}\,\, U:= U^{\geqslant 0}_q(\sl2)\Big) \notag
\end{align}
that explains the choice of elements at which we localize. Using (\ref{NAYXF}) and Proposition \ref{aA12Mar15}, we represent the spectrum $\Spec\,(\CA)$ as the disjoint union of the following subsets where we identify the sets of prime ideals via the bijections given in the statements (a) and (b) of Proposition \ref{aA12Mar15}:
\begin{align}
\Spec\,(\CA)= \Spec\,(L) \,\, & \sqcup \,\, \Spec\,(U_E) \,\, \sqcup \,\, \Spec\,(\mathbb{Y}) \notag\\
 & \sqcup \,\, \Spec\,\Big((\CA/(X))_Z\Big) \,\, \Spec\,\Big(\CA_X/(\varphi)_X \Big)\,\, \sqcup \,\, \Spec\,(\CA_{X, \varphi}).
\label{NAYXF1} 
\end{align}
A prime ideal $\gp$ of an algebra $A$ is called a \emph{completely prime ideal} if $A/\gp$ is a domain. We denote by  $\Spec_c(A)$ the set of completely prime ideals of $A$, it is called the \emph{completely prime spectrum} of $A$.
The theorem below gives an explicit description of the prime ideals of the algebra $\CA$ together with inclusions of prime ideals.
\begin{theorem}\label{A23Feb15} 
The prime spectrum $\Spec\,(\CA)$ of the algebra $\CA$ is the disjoint union of sets (\ref{NAYXF1}). More precisely,
\begin{align}
\begin{tikzpicture}[semithick, scale=0.6, text centered]
\node (YEK) at (0,8) {\small $\Big\{(Y, E, \gp)\,| \, \gp \in \Max\,(\mK[K, K^{-1}])\Big\}$};
\node (YE) at (0, 6)  {\small $(Y, E)$};
\node (Y)  at (-1, 4)  {\small $(Y)$};
\node (XZ)  at (-4, 4)  {\small $\Big\{(X, \gq)\,|\,  \gq \in \Max\,(\mK[Z])\setminus \{(Z) \} \Big\}\quad \quad \quad \quad \qquad \qquad \qquad $};
\node (E)  at (1, 4) {\small $(E)$};
\node (VC) at (4, 4) {\small $\qquad \qquad \qquad \qquad \qquad \Big\{ (\varphi, \gr)\,|\, \gr \in \Max\,(\mK[C])\setminus \{(C)\} \Big\}$};
\node (X) at (-2, 2)  {\small $(X)$};
\node (V) at (2, 2) {\small $(\varphi)$};
\node (0) at (0, 0)  {\small $0$};
\draw [ shorten <=-2pt, shorten >=-2pt] (YEK)--(YE)--(Y)--(X)--(0)--(V);
\draw [ shorten <=-2pt, shorten >=-2pt] (XZ)--(X)--(E)--(V)--(VC);
\draw [ shorten <=-2pt, shorten >=-2pt] (YE)--(E);
\draw [ shorten <=-2pt, shorten >=-2pt] (Y)--(V);
\end{tikzpicture} \label{SpecCA} 
\end{align}
where
\begin{enumerate} 
\item $\Spec\,(L) = \{ (Y, E, \gp)\,|\, \gp \in \Spec\,(\mK[K^{\pm 1}]) \}$ and $\CA/(Y, E, \gp) \simeq L/\gp$ where $L = \mK[K^{\pm 1}]$.
\item $\Spec\,(U_E)= \{(Y) \}$, $(Y)= (X,Y)= (X, Z, Y)$ and $\CA/(Y) \simeq \mK[E][K^{\pm 1};\tau]$ where $\tau(E)= q^2 E.$
\item $\Spec\,(\mathbb{Y})= \{ (E) \}$, $(E)= (E, X)= (E, \varphi)$ and $\CA/(E) \simeq \mK[Y][K^{\pm 1}; \tau]$ where $\tau(Y)=q^{-1}Y.$
\item $\Spec\,\Big((\CA/(X))_Z\Big)= \{ (X), (X, \gq)\,|\, \gq \in \Max(\mK[Z])\setminus \{ (Z)\} \}, Z= \varphi YK^{-1}$,
   \begin{enumerate} 
   \item $\CA/(X) \simeq \mathbb{E}/(X)[K^{\pm 1}; \tau]$ is a domain (see (\ref{EEX4})), and
   \item $\CA/(X, \gq) \simeq \frac{\CA_{\varphi, Y}}{(X, \gq)_{\varphi, Y}} \simeq L_{\gq} \otimes \mathbb{Y}$ is a simple domain which is a tensor product of algebras where $L_{\gq}:= \mK[Z]/\gq$ is a finite field extension of $\mK$.
   \end{enumerate}
\item $\Spec\,(\frac{\CA_X}{(\varphi)_X})= \{ (\varphi), (\varphi, \gr) \,|\, \gr \in \Max(\mK[C])\setminus \{(C)\} \}, C = XYK$, $\frac{\CA_X}{(\varphi)_X} \simeq \frac{\CA_{X, Y}}{(\varphi)_{X,Y}} \simeq \mK[C^{\pm 1}]\otimes \mathbb{Y}$,
       \begin{enumerate}
       \item $\CA/(\varphi) = \mathbb{E}/(\varphi)[K^{\pm 1}; \tau]$ is a domain (see (\ref{EEP3})).
       \item $\CA/(\varphi, \gr) \simeq \frac{\CA_{X,Y}}{(\varphi, \gr)_{X,Y}} \simeq L_{\gr} \otimes \mathbb{Y}$ is a simple domain which is a tensor product of algebras where $L_{\gr}:= \mK[C]/\gr$ is a finite field extension of $\mK$.
      \end{enumerate}
\item $\Spec\,(\CA_{X, \varphi})=\{0 \}.$
\end{enumerate}
\end{theorem}
\begin{proof}
As it was already mentioned above, we identify the sets of prime ideals via the bijection given in the statements (a) and (b) of Proposition \ref{aA12Mar15}. Recall that the set $\CS_X= \{X^i \,|\, i \in \N \}$ is a left and right denominator set of $\CA$ and $\CA_X:= \CS^{-1}_X \CA \simeq \CA \CS^{-1}_X$ is a Noetherian domain (since $\CA$ is so). The element $X$ is a normal element of $\CA$. By Proposition \ref{aA12Mar15},
\begin{align}
\Spec\,(A)= \Spec\,(\CA, X) \,\, \sqcup \,\, \Spec\,(\CA_X) \label{AXX}
\end{align}
and none of the ideals of the set $\Spec\,(\CA, X)$ is contained in an ideal of the set $\Spec\,(\CA_X)$. Similarly, the element $\varphi$ is a normal element of $\CA_X$ and, by Proposition \ref{aA12Mar15}
\begin{align}
\Spec\,(\CA_X) = \Spec\,\bigg(\Big(\CA/(\varphi)\Big)_X\bigg) \,\, \sqcup \,\, \Spec\,(\CA_{X, \varphi}). \label{AXX1} 
\end{align}
By Lemma \ref{b24May15}, the algebra $\CA_{X, \varphi}$ is a simple domain. Hence, $\Spec\,(\CA_{X, \varphi})= \{0 \},$ and statement 6 is proved.

(i) $\frac{\CA_X}{(\varphi)_X} \simeq \frac{\CA_{X, Y}}{(\varphi)_{X,Y}} \simeq \mK[C^{\pm 1}]\otimes \mathbb{Y}$: The second isomorphism holds, by (\ref{EEP5}). Using the equalities $\varphi = (q^{-1}-q)YE + X = (1-q^2)EY+ q^2 X$ we see that the elements $Y$ and $E$ are invertible in the algebra $\frac{\CA_X}{(\varphi)_X}$, and so the first isomorphism holds.

(ii) $\CA/(\varphi, \gr) \simeq \frac{\CA_{X,Y}}{(\varphi, \gr)_{X,Y}} \simeq L_{\gr} \otimes \mathbb{Y}$ \emph{for all prime ideals} $\gr \in \Max\,(\mK[C])\setminus \{(C) \}$: Since $\gr \neq (C)$, the non-zero element $C= XYK \in L_{\gr}$ is invertible in the field $L_{\gr}$. Hence, the elements $X$ and $Y$ are invertible in the algebra $\CA/(\varphi, \gr)$. Hence,
\begin{align}
\CA/(\varphi, \gr) = \frac{\CA_{X,Y}}{(\varphi, \gr)_{X,Y}}. \label{AXX2} 
\end{align}
Now, the statement (ii) follows from (\ref{AXX2}) and the statement (i).

(iii) \emph{Statement 5 holds}: Recall that the algebra $\mathbb{Y}$ is a central simple algebra. By the statement (i), the set $\Spec\,(\CA_X/ (\varphi)_X)$, as a subset of $\Spec\,(\CA)$, is equal to $\{ \CA \cap (\varphi)_X, \, \CA \cap (\varphi, \gr)_X \,|\, \gr \in \Max(\mK[C])\setminus \{(C)\} \}$. It remains to show that $\CA \cap (\varphi)_X = (\varphi)$ and $\CA \cap (\varphi, \gr)_X= (\varphi, \gr)$. The second equality follows from the statement (ii) since $(\varphi, \gr) \subseteq \CA \cap (\varphi, \gr)_X \subseteq \CA \cap (\varphi, \gr)_{X, Y} \stackrel{\rm(ii)}{=} (\varphi, \gr)$. Since $\mathbb{Y}$ is a central simple algebra, the statement (b) of statement 5 follows. The statement (a) of statement 5 is obvious (see (\ref{EEP3})). Hence, $(\varphi)= \CA \cap (\varphi)_X$, by Proposition \ref{a22May15}. So, statement 5 holds.

By Proposition \ref{aA12Mar15},
\begin{align}
\Spec\,\Big(\CA/(X)\Big) = \Spec\,\Big(\CA/(X, Z)\Big) \,\, \sqcup \,\, \Spec\,\Big(\Big(\CA/(X)\Big)_Z\Big) \label{AXX3} 
\end{align}
and none of the ideals of the set $\Spec\,(\CA/(X, Z))$ is contained in an ideal of the set $\Spec\,((\CA/(X))_Z)$.

(iv) $\CA/(X, \gq) \simeq \frac{\CA_{Y, \varphi}}{(X, \gq)_{Y, \varphi}} \simeq L_{\gq} \otimes \mathbb{Y}$ \emph{is a simple domain for all prime ideals} $\gq \in \Max(\mK[Z]) \setminus \{(Z) \}$: Since $\gq \neq (Z)$, the non-zero element $Z= \varphi YK^{-1} \in L_{\gq}$ is invertible in the field $L_{\gq}$. Hence, the elements $\varphi$ and $Y$ are invertible in the algebra $\CA/(X, \gq)$. Therefore,
\begin{align}
\CA/(X, \gq) \simeq \frac{\CA_{Y, \varphi}}{(X, \gq)_{Y, \varphi}}. \label{AXX4} 
\end{align}
Now, by (\ref{EEX6}), the statement (iv) holds.

(v) \emph{Statement 4 holds}: The algebra $\mathbb{Y}$ is a central simple algebra. By (\ref{EEX6}), the set $\Spec\,((\CA/(X))_Z)$, as a subset of $\Spec\,(\CA)$, is equal to $\{ \CA \cap (X)_{Y, \varphi}, \CA \cap (X, \gq)_{Y, \varphi} \,|\, \gq \in \Max(\mK[Z])\setminus \{(Z)\} \}$. We have to show that $\CA \cap (X)_{Y, \varphi}= (X)$ and $\CA \cap (X, \gq)_{Y, \varphi} = (X, \gq)$. The last equality follows from the statement (iv) (the algebra $\CA/(X, \gq)$ is simple and $(X, \gq) \subseteq \CA \cap (X, \gq)_{Y, \varphi} \subsetneqq \CA$, hence $(X, \gq)= \CA \cap (X, \gq)_{Y, \varphi}$). Now, the statement (b) of statement 4 holds. The statement (a) is obvious, see (\ref{EEX4}). Hence, $(X)= \CA \cap (X)_{Y, \varphi}$, by Proposition \ref{a22May15}. So, the proof of the statement (v) is complete.

In the algebra $\CA$, using the equality $\varphi = X + (q^{-1}-q)YE$ we see that
\begin{align}
Z \equiv \varphi YK^{-1} \equiv (q^{-1}-q)YEYK^{-1} \equiv (1-q^2)EY^2K^{-1} \mod (X). \label{ZEYK1} 
\end{align}

(vi) $(Y)= (X, Y)= (X, Z, Y)$: The first equality follows from the relation $X= EY - q^{-1}YE$. Then the second equality follows from (\ref{ZEYK1}).

(vii) $(E)= (E, X) = (E, \varphi)$: The first equality follows from the relation $X= EY - q^{-1}YE$. Then the second equality follows from the  definition of the element $\varphi = X + (q^{-1}-q)YE.$

(viii) \emph{The elements $Y$ and $E$ are normal in  $\CA/(X)$}: The statement follows from (\ref{EEX3}) and (\ref{EEX4}).

(ix) $\bigg(\frac{\CA}{(X, Z)}\bigg)_Y \simeq \bigg(\frac{\CA}{(E)}\bigg)_Y \simeq \mathbb{Y}$ \emph{is a simple domain}: By (\ref{ZEYK1}), $(X, Z)= (X, EY^2) \subseteq (X, E) \stackrel{\rm (vii)}{=} (E)$, hence $(X, EY^2)_Y = (X, E)_Y = (E)_Y$, by the statement (vii). Now, by (\ref{EEX4}), $\bigg(\frac{\CA}{(X, Z)}\bigg)_Y \simeq \frac{\CA_Y}{(X, Z)_Y}\simeq \frac{\CA_Y}{(E)_Y} \simeq \bigg(\frac{\CA}{(X, E)}\bigg)_Y \simeq \mathbb{Y}$.

By the statement (viii) and Proposition \ref{aA12Mar15},
\begin{align}
\Spec\,\Big(\CA/(X, Z)\Big)= \Spec\,\Big(\CA/(X, Z,Y)\Big) \,\, \sqcup \,\, \Spec\,\Big(\Big(\CA/(X, Z)\Big)_Y \Big). \label{AXX5} 
\end{align}
By the statement (vi), $\CA/(X, Z, Y) \simeq \CA/(X,Y) \simeq U:= U^{\geqslant 0}_q(\sl2)$. By the statement (ix),  $\Big(\CA/(X,Z)\Big)_Y \simeq \CA_Y / (E)_Y \simeq \mathbb{Y}$ is a simple domain. So, the set $\Spec\,\Big(\Big(\CA/(X, Z)\Big)_Y\Big),$ as a subset of $\Spec\,(\CA)$, consists of the ideal $(E)$. In more details, since $(X, Z) \subseteq (E), (X, Z)_Y = (E)_Y$ (see the proof of the statement (ix)) and $\CA/(E) = \CA/(E, X) = \mK[Y][K^{\pm 1}; \tau]$ is a domain, the result follows from Proposition \ref{a22May15}. So, statement 3 holds.

The element $E$ is a normal element of the algebra $U$. By Proposition \ref{aA12Mar15},
\begin{align}
\Spec\,(U)= \Spec\,(U/(E)) \,\, \sqcup \,\, \Spec\,(U_E). \label{AXX6} 
\end{align}
Since $L = U/(E)$, statement 1 follows. The algebra $U_E \simeq \mK[E^{\pm 1}][K^{\pm 1};\tau]$ is a central simple domain. Since $U = \CA/(Y) = \CA/(X, Z,Y)$ (the statement (vi)) is a domain, the set $\Spec\,(U_E)$, as a subset of $\Spec\,(\CA)$, consists of a single ideal $(Y)$, and statement 2 follows.

We proved that (\ref{NAYXF1}) holds. Clearly, we have the inclusions as on the diagram (\ref{SpecCA}). It remains to show that there is no other inclusions. The ideals $(Y, E, \gp), (X, \gq)$ and $(\varphi, \gr)$ are the maximal ideals of the algebra $\CA$ (see statement 1, 4, and 5). By (\ref{AXX4}) and the relations given in (\ref{SpecCA}), there are no additional lines leading to the maximal ideals $(X, \gq)$. Similarly, by (\ref{AXX2}) and the relations given in (\ref{SpecCA}), there are no additional lines leading to the maximal ideals $(\varphi, \gr)$. The elements $X$ and $\varphi$ are normal elements of the algebra $\CA$ such that $(X) \not\subseteq (\varphi)$ and $(X) \not\supseteq (\varphi)$, by (\ref{EEX1}). The proof of the theorem is complete.
\end{proof}

For an algebra $A$, $\Max(A)$ is the set of its maximal ideals. The next corollary is an explicit description of the set $\Max(\CA)$.
\begin{corollary}\label{a24Feb15} 
$\Max\,(\CA) = \CP \,\, \sqcup \,\, \CQ\,\, \sqcup \,\,\CR$ where $\CP:=
\Big\{(Y, E, \gp)\,| \, \gp \in \Max\,(\mK[K, K^{-1}]) \Big\}, \,\, \CQ:= \Big\{(X, \gq)\,|\,  \gq \in \Max\,(\mK[Z])\setminus \{(Z)\} \Big\}$ and $\CR:= \Big\{ (\varphi, \gr)\,|\, \gr \in \Max\,(\mK[C])\setminus \{(C)\} \Big\}.$
\end{corollary}
\begin{proof}
The corollary follows from (\ref{SpecCA}).
\end{proof}

Let $M$ be an $\CA$-module, $a \in \CA$ and $a_M \cdot : M \rightarrow M, \,\, m \mapsto am$. The ideal of $\CA$, $\ann_{\CA}(M):= \{ a \in \CA\,|\, aM=0 \}$ is called the \emph{annihilator} of $M$. The $\CA$-module is called \emph{faithful} if it has zero annihilator. The next corollary is a faithfulness criterion for simple $\CA$-modules.

\begin{corollary}\label{b29Mar15} 
Let $M$ be a simple $\CA$-module. Then $M$ is a faithful $\CA$-module iff $\ker(X_M \cdot)= \ker(\varphi_M \cdot) =0$ iff $M_X \neq 0$ and $M_{\varphi} \neq 0$ (where $M_X$ and $M_{\varphi}$ are the localizations of the $\CA$-module $M$ at the powers of the elements $X$ and $\varphi$, respectively).
\end{corollary}
\begin{proof}
The $\CA$-module is simple, so $\ann_{\CA}(M) \in \Spec\,(\CA)$. The elements $X$ and $\varphi$ are normal elements of the algebra $\CA$. So, $\ker(X_M \cdot)$ and $\ker(\varphi_M \cdot)$ are submodules of $M$. Either $\ker(X_M\cdot)=0$ or $\ker(X_M \cdot)=M$, and in the second case $\ann_{\CA}(M) \supseteq (X)$. Similarly, either $\ker(\varphi_M \cdot)=0$ or $\ker(\varphi_M \cdot)=M$, and in the second case $\ann_{\CA}(M) \supseteq (\varphi)$. Conversely, if $\ann_{\CA}(M) =0$ then $\ker(X_M \cdot)= \ker(\varphi_M \cdot)=0$. If $\ker(X_M \cdot)=\ker(\varphi_M \cdot)=0$ then $\ann_{\CA}(M) =0$, by (\ref{SpecCA}). So, the first `iff' holds.

For a normal element $u = X, \varphi$, $\ker(u_M \cdot)=0$ iff $M_u \neq 0$. Hence, the second `iff' follows.
\end{proof}

Let $A$ be an algebra. The annihilator of each simple $A$-module is a prime ideal. Such prime ideals are called \emph{primitive} and the set $\Prim\,(A)$ of all of them is called the \emph{primitive spectrum} of $A$.
\begin{proposition} \label{a27Feb15} 
$\Prim(\CA)= \Max(\CA) \,\sqcup\, \{ (Y), \,\, (E), \,\, 0 \}$.
\end{proposition}
\begin{proof}
Clearly, $\Prim(\CA) \supseteq \Max(\CA)$. The ideals $(X), \,(\varphi)$ and $(Y, E)$ are not primitive ideals as the corresponding factor algebras contain the central elements $Z, C$ and $K$, respectively.

(i) Let us show that $(Y) \in \Prim(\CA)$. For $\l \in \mK^*$, let $I_{\l} = (Y)+ \CA(E-\l)$. Since $\CA/(Y) \simeq U$, the left $\CA$-module $M(\l):= \CA/I(\l) \simeq U/U(E-\l) \simeq \mK[K^{\pm 1}] \bar{1}$ is a simple $\CA$-module/$U$-module where $\bar{1}= 1+ I(\l)$. By the very definition, the prime ideal $\ga:= \ann_{\CA}(M(\l))$ contains the ideal $(Y)$ but does not contain the ideal $(Y, E)$ since otherwise we would have $0= E \bar{1} = \l \bar{1}  \neq 0$, a contradiction. By (\ref{SpecCA}), $\ga = (Y).$

(ii)  Let us show that $(E) \in \Prim(\CA)$. By Theorem \ref{A23Feb15}, $(E)=(E,X)$ and $\bar{\CA}:= \CA/(E) \simeq \mK[Y][K^{\pm 1};\sigma]$ where $\sigma(Y)=q^{-1}Y.$ For $\l \in \mK^*$, the $\CA$-module $T(\l):= \bar{\CA}/\bar{\CA}(Y-\l) \simeq \mK[K^{\pm 1}]\bar{1}$ is a simple module (since $q$ is not a root of 1), where $\bar{1}= 1+ \bar{\CA}(Y-\l)$. Clearly, the prime ideal $\gb:= \ann_{\CA}(T(\l))$ contains the ideal $(E)$ but does not contain the ideal $(Y, E)$ since otherwise we would have $0= Y \bar{1} = \l \bar{1}\neq 0$, a contradiction. By (\ref{SpecCA}), $\gb= (E)$.

(iii) By Theorem \ref{25Mar15}, $0$ is a primitive ideal of $\CA$.
\end{proof}


\begin{corollary}\label{a13Mar15} 
Every prime ideal of the algebra $\CA$ is completely prime, i.e., $\Spec_c(\CA)= \Spec\,(\CA)$.
\end{corollary}
\begin{proof}
See Theorem \ref{A23Feb15}.
\end{proof}

\section{The automorphism group of $\CA$} \label{AutoGroup} 
In this section, the group $G:= \Aut_{\mK}(A)$ of automorphisms of the algebra $\CA$ is found (Theorem \ref{B27Feb15}). Corollary \ref{b27Feb15} describes the orbits of the action of the group $G$ on $\Spec\,(\CA)$ and the set of fixed points.

We introduce a degree filtration on the algebra $\CA$ by setting $\deg(K)= \deg(K^{-1})=0$ and $\deg(E)= \deg(X)= \deg(Y)=1$. So, $\CA= \bigcup_{n \in \N} \CA[n]$ where $\CA[n]= \sum \mK X^i Y^j E^k K^l$ with $\deg(X^i Y^j E^k K^l):= i+j+k \leqslant n$. Let $\Gr\, \CA:= \bigoplus_{i \in \N}\CA[i]/\CA[i-1]$ (where $\CA[-1]:=0$) be the associated graded algebra of $\CA$ with respect to the filtration $\{ \CA[n] \}_{n \geqslant 0}$. For an element $a \in \CA$, we denote by $\Gr\,a \in \Gr\, \CA$ the image of $a$ in $\Gr\, \CA$. It is clear that $\Gr\, \CA$ is an iterated Ore extension, $\Gr\,\CA \simeq \mK[X][Y;\alpha][E;\beta][K^{\pm 1};\gamma]$ where $\alpha(X)=q^{-1}X,\, \beta(X)=qX,\, \beta(Y)=q^{-1}Y, \,\gamma(X)=qX, \, \gamma(Y)=q^{-1}Y$ and $\gamma(E)=q^2E$. In particular, $\Gr\, \CA$ is a Noetherian domain of Gelfand-Kirillov dimension $\GK(\Gr\, \CA)=4$ and the elements $X, Y$ and $E$ are normal in $\Gr\,\CA$.

The group of units $\CA^*$ of the algebra $\CA$ is equal to $\{ \mK^* K^i \,|\, i\in \Z \}= \mK^* \times \langle K \rangle$ where $\langle K \rangle = \{ K^i \,|\, i \in \Z \}$. The next theorem is an explicit description of the group $G$.

\begin{theorem}\label{B27Feb15} 
$\Aut_{\mK}(\CA)= \{ \sigma_{\l, \mu, \gamma, i} \,|\, \l, \mu, \gamma \in \mK^*, \, i \in \Z \} \simeq (\mK^*)^3 \rtimes \Z$ where $\sigma_{\l, \mu,\gamma, i}: X \mapsto \l K^i X, \, Y \mapsto \mu K^{-i}Y, \, K \mapsto \gamma K, \, E \mapsto \l \mu^{-1} q^{-2i} K^{2i}E$ (and $\sigma_{\l, \mu, \gamma,i}(\varphi)= \l K^i \varphi$). Furthermore, $\sigma_{\l, \mu, \gamma, i} \sigma_{\l',\mu', \gamma', j}= \sigma_{\l \l' \gamma^j, \mu \mu' \gamma^{-j}, \gamma \gamma',  i+j}$ and $\sigma^{-1}_{\l, \mu, \gamma, i}= \sigma_{\l^{-1}\gamma^i, \mu^{-1}\gamma^{-i}, \gamma^{-1}, -i}.$
\end{theorem}
\begin{proof}
Using the defining relations of the algebra $\CA$, one can verify that $\sigma_{\l, \mu, \gamma,i} \in G$ for all $\l, \mu, \gamma \in \mK^*$ and $i \in \Z$. The subgroup $G'$ generated by these automorphisms is isomorphic to the semi-direct product $(\mK^*)^3 \rtimes \Z.$ It remains to show that $G= G'$. Recall that the elements $X$ and $\varphi$ are normal in the algebra $\CA$. Let $\sigma \in G$, we have to show that $\sigma \in G'.$

By (\ref{SpecCA}), there are two options either the ideals $(X)$ and $(\varphi)$ are $\sigma$-invariant or, otherwise, they are interchanged. In more details, either, for some elements $\l, \l' \in \mK^*$ and $i,j \in \Z$,
\begin{enumerate}[label=(\alph*)]
\item $\sigma(X) = \l K^i X$ and $\sigma(\varphi)= \l' K^j \varphi$, or, otherwise,
\item $\sigma(X)= \l K^i \varphi$ and $\sigma(\varphi)= \l' K^j X.$
\end{enumerate}

(i) \emph{$\sigma(K)= \gamma K$ for some $\gamma \in \mK^*$}: The group of units $\CA^*$ of the algebra $\CA$ is equal to $\{ \gamma K^s \,|\, \gamma \in \mK^*, \, s \in \Z \}$. So, either $\sigma(K)= \gamma K$ or, otherwise, $\sigma(K)= \gamma K^{-1}$ for some $\gamma \in \mK^*$. Let us show that the second case is not possible. Notice that $KX = q XK$ and $K \varphi = q \varphi K$, i.e., the elements $X$ and $\varphi$ have the same commutation relation with the element $K$.  Because of that it suffices to consider one of the cases (a) or (b) since then the other case can be treated similarly. Suppose that the case (a) holds and that $\sigma(K)=\gamma K^{-1}$. Then the equality $\sigma(K)\sigma(X)= q \sigma(X) \sigma(K)$ yields the equality $\gamma K^{-1} \cdot \l K^i X = q \l K^i X \cdot \gamma K^{-1}= q\l \gamma q K^{i-1}X$. Hence, $q^2 =1$, a contradiction.

(ii) \emph{$\deg \sigma(Y)= \deg \sigma(E)=1$}: Recall that $YE= q'^{-1}(\varphi - X)$ where $q':= q^{-1}-q.$ By applying $\sigma$ to this equality we obtain the equality $\sigma(Y)\sigma(E)= q'^{-1} \sigma(\varphi -X)$. Hence, $\deg\Big(\sigma(Y) \sigma(E)\Big)= \deg\Big(q'^{-1}\sigma(\varphi -X)\Big)=2$, in both cases (a) and (b). Thus, there are three options for the pair $\Big(\deg\sigma(Y), \deg \sigma(E) \Big)$: $(1, 1), \,\, (0, 2)$ or $(2, 0)$. The last two options are not possible since otherwise we would have $\sigma(Y) \in \mK[K^{\pm 1}]$ or $\sigma(E) \in \mK[K^{\pm 1}]$, respectively. Hence, $\sigma(Y)\sigma(K)= \sigma(K) \sigma(Y)$ or $\sigma(E) \sigma(K)= \sigma(K) \sigma(E)$, respectively. But this is impossible since $YK \neq KY$ and $EK \neq KE$. Therefore, $\deg \sigma(Y)= \deg \sigma(E)=1.$

(iii) \emph{$\sigma(Y)=aY$ and $\sigma(E)=bE$ for some nonzero elements $a, b \in \mK[K^{\pm 1}]$}: Applying $\sigma$ to the relation $KEK^{-1} = q^2 E$ we obtain the equality $K \sigma(E)K^{-1} = q^2 \sigma(E)$. Since $\deg \sigma(E)=1$, $\sigma(E)= bE + u X + v Y + w$ for some elements $b, u, v, w \in \mK[K^{\pm 1}]$. Using the relations $KXK^{-1}= qX$ and $KYK^{-1}= q^{-1}Y$, we see that $u=v=w=0$, i.e., $\sigma(E)= bE$. Similarly, applying $\sigma$ to the relation $KYK^{-1}= q^{-1}Y$, we obtain the equality $K \sigma(Y) K^{-1} = q^{-1} \sigma(Y)$. Since $\deg \sigma(Y)=1$, $\sigma(Y)= a Y + u' E + v' X + w'$ for some elements $a, u', v', w' \in \mK[K^{\pm 1}]$. Using the relations $KXK^{-1} = qX, \,\, KYK^{-1}= q^{-1}Y$ and $KEK^{-1}= q^2 E$, we see that $u'= v'= w'=0,$ i.e., $\sigma(Y)= a Y.$

(iv) \emph{$i=j$ (see the cases (a) and (b))}: The elements $X$ and $\varphi$ commute, hence $\sigma(X)\sigma(\varphi)= \sigma(\varphi) \sigma(X)$. Substituting the values of $\sigma(X)$ and $\sigma(\varphi)$ into this equality yields $q^{-i}= q^{-j}$ in both cases (a) and (b), i.e., $i=j$ (since $q$ is not a root of unity).

(v) \emph{The case (b) is not possible}: Suppose that the case (b) holds, i.e., $\sigma(X)= \l K^i \varphi$ and $\sigma(\varphi) = \l' K^i X$ (see the statement (iv)), we seek a contradiction. To find the contradiction we use the relations $qYX = XY$ and $Y \varphi = q \varphi Y$. Applying the automorphism $\sigma$ to the first equality gives $\sigma(qYX)= q aY \cdot \l K^i \varphi = q a\l q^i K^i Y \varphi $ and $\sigma(XY)= \l K^i \varphi \cdot a Y = \l K^i \tau(a) \varphi Y = \l K^i \tau(a) q^{-1} Y \varphi $ where $\tau$ is the automorphism of the algebra $\mK[K^{\pm 1}]$ given by the rule $\tau(K)= q^{-1}K$. Hence, $\tau(a)= q^{i+2}a$, i.e., $a = \xi K^{-i-2}$ for some $\xi \in \mK^*$. So, $\sigma(Y)= \xi K^{-i-2}Y$. Now applying $\sigma$ to the second equality, $Y \varphi = q \varphi Y$, we have the equalities $\sigma(Y \varphi)= \xi K^{-i-2}Y \cdot \l' K^i X = \xi \l' q^i K^{-2} YX$ and $\sigma(q\varphi Y)= q\l' K^i X \cdot \xi K^{-i-2} Y = \xi \l' q^{i+3} K^{-2} XY = \xi \l' q^{i+4} K^{-2} YX$. Therefore, $q^4 =1$, a contradiction (since $q$ is not a root of unity). This means that the only case (a) holds. Summarizing, we have $\sigma(X)= \l K^i X, \,\, \sigma(K)= \gamma K,\,\, \sigma(\varphi)= \l' K^i \varphi, \,\, \sigma(Y)= aY, \,\, {\rm and}\,\, \sigma(E)=bE.$

(vi) \emph{$a = \mu K^{-i}$ for some $\mu \in \mK^*$ (i.e., $\sigma(Y)= \mu K^{-i} Y$)}: Applying the automorphism $\sigma$ to the relation $qYX =XY$ yields: $\sigma(qYX)= q aY \cdot \l K^i X = \l a q^i K^i qYX = \l a q^i K^i XY$ and $\sigma(XY)= \l K^i X  \cdot aY = \l K^i \tau(a) XY$. Therefore, $\tau(a)= q^i a$, i.e., $a = \mu K^{-i}$ for some element $\mu \in \mK^*$.

(vii) \emph{$b= \d K^{2i}$ for some $\d \in \mK^*$ (i.e., $\sigma(E)= \d K^{2i} E$)}: Applying the automorphism $\sigma$ to the relation $q XE = EX$ yields: $\sigma(qXE)= q \l K^i X \cdot bE = \l K^i \tau(b) q XE = \l K^i \tau(b) EX$ and $\sigma(EX)= bE \cdot \l K^i X = \l K^i b q^{-2i} EX$. Therefore, $\tau(b)= q^{-2i}b$, i.e., $b = \d K^{2i}$ for some $\d \in \mK^*$.

(viii)\emph{ $\d = \l \mu^{-1} q^{-2i}$}: Applying the automorphism $\sigma$ to the relation $EY = X+ q^{-1}YE$ gives: $\sigma(EY)= \d K^{2i} E \cdot \mu K^{-i}Y = \d \mu q^{2i} K^i EY = \d \mu q^{2i} K^i (X + q^{-1}YE)$ and $\sigma(X+q^{-1}YE)= \l K^i X + q^{-1} \mu K^{-i} Y \cdot \d K^{2i} E = K^i (\l X + \d \mu q^{2i} q^{-1}YE)$. Therefore, $\d \mu q^{2i}= \l$, and the statement (viii) follows. The proof of the theorem is complete.
\end{proof}

\begin{corollary}\label{b27Feb15} 

\
\begin{enumerate}
\item The prime ideals $\CI:= \{ 0, (X), \,(\varphi), \, (Y), \, (E), \, (Y, E) \}$ are the only prime ideals of $\CA$ that are invariant under action of the group $G$ of automorphisms of $\CA$ (i.e., $G \gp = \{\gp \}$ for all $\gp \in \CI$).

\item If, in addition, $\mK$ is an algebraically closed field, then each of the three series of prime ideals in $\Spec(\CA)$ is a simple $G$-orbit. In particular, there are $9$ $G$-orbits in $\Spec(\CA)$.
\end{enumerate}
\end{corollary}
\begin{proof}
By Theorem \ref{B27Feb15}, $G\gp = \{\gp \}$ for all $\gp \in \CI$. Let $\sigma= \sigma_{\l, \mu, \gamma, i}$, then $\sigma(Z)= \l \mu \gamma^{-1} q^i Z, \, \sigma(C)= \l \mu \gamma q^i C$ and $\sigma(K)=\gamma K$. Now the corollary is obvious since $\gp = (K-\alpha), \,\gq = (Z-\beta)$ and $\gr = (C- \gamma)$ for some $\alpha, \beta,\gamma \in \mK^*$.
\end{proof}

\section{Classification of simple unfaithful $\CA$-modules} \label{ClaANNz} 
The aim of this section is to classify all simple $\CA$-modules with non-zero annihilators. The set of simple $\CA$-modules are partitioned according to their annihilators which were described in Proposition \ref{a27Feb15}. For each primitive ideal $\gp$ of $\CA$, the factor algebra $\CA/\gp$ is described by Theorem \ref{A23Feb15}. They are of the type $R=D[t, t^{-1};\sigma]$ where $D$ is a commutative Dedekind domain and the automorphism $\sigma$ of $D$ are of two types (I) or (F), see below. For our purposes we need only (I) $D= \mK[H, H^{-1}]$, (F) $D=\mK[H]$ and in both cases $\sigma(H)=qH$ where $q \in \mK^*$ is not a root of 1.\\

\noindent\textbf{Classification of simple $R$-modules.}
Let us recall the classification of simple $R$-modules where $R=D[t,t^{-1};\sigma]$, $D$ is a commutative Dedekind domain and $\sigma$ satisfies one of the two conditions ((I) or (F), see below). These results are very particular cases of classification of simple modules over a GWA $D(\sigma, \,\,a)$ obtained in \cite{Bav-SimGWA-1992,Bav-GWArep,Bav-THM-GLGWA-1996,Bav-SimpModQuanPlane,Bav-OystGeneCrossPro,Bav-OreSim-1999,Bav-OystWit-Woron}.

Let $\CS=D\setminus \{0 \}$ and $k = \CS^{-1}D$ be the field of fractions of the ring $D$. The ring $R$ is a subring of $B:= k[t, t^{-1};\sigma]$. The ring $B= \CS^{-1}R$ is a (left and right) localization of $R$ at $\CS$. The ring $B$ is a Euclidean ring and so is a principle left and right ideal domain. Every simple $B$-module is isomorphic to $B/Bb$ for some \emph{irreducible} element $b \in B$ (i.e., $b=ac$ implies either $a$ or $c$ is a unit in $B$). Two simple $B$-modules $B/Ba$ and $B/Bb$ are isomorphic iff the irreducible elements $a$ and $b$ are \emph{similar}, i.e., there exists an element $c \in B$ such that $1$ is the greatest common right divisor of $b$ and $c$, and $ac$ is the least common left multiple of $b$ and $c$.

Let $G:= \langle \sigma \rangle$ be the subgroup of $\Aut(D)$ generated by $\sigma$. The group $G$ acts on the set $\Max(D)$ of maximal ideals of $D$. For each $\gp \in \Max(D), \mathscr{O}(\gp):= \{\sigma^i(\gp)\,|\,i \in \Z \}$ is the orbit of $\gp$. The set of all $G$-orbits in $\Max(D)$ is denoted by $\Max(D)/G$. The orbit $\mathscr{O}(\gp)$ is called an \emph{infinite} or \emph{linear} orbit if $|\mathscr{O}(\gp)|= \infty$; otherwise the orbit $\mathscr{O}(\gp)$ is called a \emph{finite} or \emph{cyclic} orbit. If $\mathscr{O}(\gp)$ is an infinite orbit then the map $\Z \rightarrow \mathscr{O}(\gp),\, i \mapsto \sigma^i(\gp)$, is a bijection, and we write $\sigma^i(\gp) \leqslant \sigma^j(\gp)$ if $i \leqslant j.$ So, the total ordering of $\Z$ is passed to $\mathscr{O}(\gp)$. This ordering does not depend on the choice of the ideal $\gp$ in the orbit $\mathscr{O}(\gp)$.

Given elements $\alpha, \beta \in D\setminus \{0 \},$ we write $\alpha < \beta$ if there are no maximal ideas $\gp$ and $\gq$ that belong to the same \emph{infinite} orbit and such that $\alpha \in \gp, \beta \in \gq$ and $\gp \geqslant \gq$. In particular, if $\alpha \in D^*$ is a unit of $D$ then $\alpha < \beta$ for all $\beta \in D\setminus \{0 \}$. The relation $<$ is not a partial order on $D \setminus \{0 \}$ as $1 < 1$. Clearly, $\alpha < \beta$ iff $\sigma^j(\alpha) < \sigma^j(\beta)$ for some/all $j \in \Z$. \\

\noindent
\textbf{Definition.} An element $b= t^m \beta_m + t^{m+1}\beta_{m+1} + \ldots + t^n \beta_n \in R$ where $\beta_i \in D, \, m<n$ and $\beta_m, \, \beta_n \neq 0$ is called an $l$-\emph{normal} element if $\beta_n < \beta_m.$\\

\noindent
\textbf{Definition.} We say that the automorphism $\sigma$ of $D$ is of (I)-\emph{type} automorphism, if all $G$-orbits in $\Max(D)$ are infinite. We say that the automorphism $\sigma$ of $D$ is of (F)-\emph{type} if there is a maximal ideal $\gp \in \Max(D)$ which is $\sigma$-invariant (i.e., $\sigma(\gp)=\gp$), $\{\gp\}$ is the only finite orbit and the automorphism $\bar{\sigma}: D/\gp \rightarrow D/\gp, \, d+\gp \mapsto \sigma(d)+\gp,$ is the identity automorphism (or equivalently, the factor ring $R/\gp = D/\gp[t, t^{-1};\bar{\sigma}]$ is a \emph{commutative} ring $D/\gp[t, t^{-1}]$). The ideal $\gp$ is called the \emph{exceptional} ideal of $D$.\\

\noindent
\textbf{Examples.}
1. Let $D= \mK[H, H^{-1}]$ and $\sigma(H)=qH$ where $q \in \mK^*$ is not a root of one. The automorphism $\sigma$ is of (I)-type.

2. Let $D= \mK[H]$ and $\sigma(H)=qH$ where $q \in \mK^*$ is not a root of one. The automorphism $\sigma$ is of (F)-type.

Let $\hat{R}$ be the set of isomorphism classes $[M]$ of simple $R$-modules $M$. Then
\begin{align*}
\hat{R}= \hat{R}(D{\rm\text{-}torsion})\,\, \sqcup \,\, \hat{R}(D{\rm\text{-}torsionfree})
\end{align*}
is a disjoint union where $\hat{R}(D{\rm\text{-}torsion}):= \{[M]\in \hat{R}\,|\, \CS^{-1}M =0 \}$ and $\hat{R}(D{\rm\text{-}torsionfree}):= \{ [M] \in \hat{R}\,|\, \CS^{-1}M \neq 0 \}$. An $R$-module $M$ is called a \emph{weight} $R$-module if $M$ is a semisimple $D$-module, i.e.,
$$M= \bigoplus_{\gp \in \Max(D)} M_{\gp}$$
where $M_{\gp}= \{ v \in M \,|\, \gp v =0 \}$ is called the \emph{component} of $M$ of \emph{weight} $\gp$. The set $\Supp(M)= \{ \gp \in \Max(D)\,|\, M_{\gp} \neq 0 \}$ is called the \emph{support} of the weight $\CA$-module $M$. It is also denoted $\Supp_D(M)$ and is called $D$-\emph{support} if we want to stress over which ring $D$ we consider weight modules.
Every weight $R$-module is $D$-torsion, but not vice versa. We denote by $\hat{R}$($D$-torsion, infinite) and $\hat{R}$($D$-torsion,  finite) the sets of isomorphism classes of simple, weight $R$-modules with infinite and finite support, respectively.

The next theorem classifies the simple, $D$-torsion $R$-modules.
\begin{theorem}  \label{A15Mar15} 
\begin{enumerate}
\item \cite{Bav-GlGWA-1996} Suppose that $\sigma$ is of (I)-type. Then the map
$$\Max(D)/G \longrightarrow \hat{R}(D{\rm\text{-}torsion}), \,\, \mathscr{O}(\gp) \mapsto R/R\gp,$$
is a bijection with the inverse $M \mapsto \Supp(M)$.
\item \cite{Bav-SimpModQuanPlane,Bav-OystGeneCrossPro} Suppose that $\sigma$ is of (F)-type. Then $\hat{R}${\rm ($D$-torsion)} $= \hat{R}${\rm ($D$-torsion, infinite)}\,\,$\sqcup\,\, \hat{R}${\rm ($D$-torsion, finite)} is a disjoint union where
\begin{enumerate}
\item the map $$\Big( \Max(D)\setminus \{ \gp \}\Big)/G \longrightarrow \hat{R}(D{\rm\text{-} torsion},\,\, {\rm infinite}), \,\, \mathscr{O}(\gp) \mapsto R/R\gp,$$ is a bijection with the inverse $M \mapsto \Supp(M)$, and
\item $\hat{R}${\rm ($D$-torsion, finite)} $= \widehat{R/(\gp)}= \widehat{D/\gp[t,t^{-1}]}.$
\end{enumerate}
\end{enumerate}
\end{theorem}

The next theorem is about classification of simple, $D$-torsionfree $R$-modules.
\begin{theorem}   \label{B15Mar15} 
\begin{enumerate}
\item \cite{Bav-GlGWA-1996} Suppose that $\sigma$ is of (I)-type. Then
$$\hat{R}(D{\rm\text{-}torsionfree})=\bigg\{ [R/R\cap Bb] \,|\, b \in R \,\,{\rm is \,\, an}\,\, l{\rm\text{-}normal},\, {\rm irreducible\,\, element\, of \,} B \bigg\}.$$
The simple $R$-modules $R/R \cap Bb$ and $R/R\cap Bb'$ are isomorphic iff the $B$-modules $B/Bb$ and $B/Bb'$ are isomorphic.

\item \cite{Bav-SimpModQuanPlane,Bav-OystGeneCrossPro} Suppose that $\sigma$ is of (F)-type. Then
\begin{align*}
\hat{R}(D{\rm\text{-}torsionfree}) = \bigg\{  & [R/R\cap Bb] \,|\,  b\in R \,\,{\rm is \,\, an}\,\, l{\rm\text{-}normal}, \,{\rm irreducible \,\,element \,\, of}\,\, B\,\, {\rm such \,\, that\,\,}\\  & R =R\gp+R\cap Bb \bigg\}.
\end{align*}
The simple $R$-modules $R/R\cap Bb$ and $R/R \cap Bb'$ are isomorphic iff the $B$-modules $B/Bb$ and $B/Bb'$ are isomorphic.
\end{enumerate}
\end{theorem}

\noindent\textbf{Classification of simple $\CA$-modules with non-zero annihilators.}
The set $\hat{\CA}$ is a disjoint union
\begin{align}
\hat{\CA} = \bigsqcup_{\gp \in \Prim(\CA)} \hat{\CA}\,(\gp) \label{AAPU} 
\end{align}
where $\hat{\CA}\,(\gp):= \{[M] \in \hat{\CA}\,|\, \ann_{\CA}(M)=\gp \}$. Below, a description of the set $\hat{\CA}\,(\gp)$ is given for each nonzero primitive ideal $\gp \in \Prim(\CA)$. We use the description of $\Prim\,(\CA)$ (Proposition \ref{a27Feb15}), and so we have to consider case (a)--(e), see below.

\begin{enumerate}[label=(a)]
\item Let $\gp \in \Spec\,(\mK[K^{\pm 1}])\setminus \{0 \}$ and $P = (Y, E, \gp) \in \CP$. Then
\begin{align}
\hat{\CA}\,\Big ((Y,E,\gp)\Big)= \{ \CA/(Y, E, \gp) \simeq \mK[K^{\pm 1}]/\gp \}. \label{AAPU1} 
\end{align}
\end{enumerate}

\begin{enumerate}[label=(b)]
\item Let $\gp = (Y)$. Then $\CA/(Y)\simeq U= \mK[E][K^{\pm 1};\sigma]$ where $\sigma(E)=q^{2}E$ (Theorem \ref{A23Feb15}.(2)) and the ideal $(E)$ is the only height $1$ prime ideal of the algebra $U$ (see (\ref{SpecCA})). The element $E$ is a normal element of the algebra $U$. The algebra $U$ is the ring of (F)-type, where $(E)$ is the exceptional ideal of $\mK[E]$.
\end{enumerate}

\begin{lemma}\label{a15Mar15}
$\hat{\CA}\,\Big((Y)\Big)= \hat{U}$\,{\rm ($\mK[E]$\text{-}torsion, infinite)}\,\,$\sqcup \,\, \hat{U}$\,{\rm ($\mK[E]$\text{-}torsionfree)} (see Theorem \ref{A15Mar15}.(2)) and Theorem \ref{B15Mar15}.(2)).
\end{lemma}
\begin{proof*}
The result follows from Theorem \ref{A15Mar15}.(2), Theorem \ref{B15Mar15}.(2) and the fact that each simple module of the set $\hat{U}$\,($\mK[E]$-torsion, finite) has annihilator strictly larger than $(Y)$.
\end{proof*}

\begin{enumerate}[label=(c)]
\item Let $\gp =(E)$. Then $\CA/(E) \simeq \mK[Y][K^{\pm 1};\tau]$ where $\tau(Y) =q^{-1}Y$ (Theorem \ref{A23Feb15}.(3)). The automorphism $\tau$ of $\mK[Y]$ is of (F)-type where $(Y)$ is the exceptional ideal of $\mK[Y]$.
\end{enumerate}

\begin{lemma}\label{b15Mar15} 
Let $R= \mK[Y][K^{\pm 1};\tau]$ where $\tau(Y)= q^{-1}Y$. Then $\hat{\CA}\, \Big((E)\Big) = \hat{R}$\,{\rm ($\mK[Y]$-torsion, infinite)} $\sqcup\,\, \hat{R}$\,{\rm ($\mK[Y]$-torsionfree)} (see Theorem \ref{A15Mar15}.(2) and Theorem \ref{B15Mar15}.(2)).
\end{lemma}
\begin{proof*}
The result follows from Theorem \ref{A15Mar15}.(2), Theorem \ref{B15Mar15}.(2) and the fact that each simple module of the set $\hat{R}$\,($\mK[Y]$-torsion, finite) has annihilator strictly larger than $(E)$.
\end{proof*}

\begin{enumerate}[label=(d)]
\item Let $\gp = (X, \gq)$ where $\gq \in \Max\,(\mK[Z])\setminus \{(Z) \}$. By Theorem \ref{A23Feb15}.(4), $\CA/(X, \gq) \simeq L_{\gq}[Y^{\pm 1}][K^{\pm 1};\tau]$ where $\tau(Y)=q^{-1}Y$ and so the automorphism $\tau$ is of (I)-type.
\end{enumerate}

\begin{lemma}\label{c15Mar15} 
Let $R= L_{\gq}[Y^{\pm 1}][K^{\pm 1};\tau]$ where $\tau(Y)=q^{-1}Y$. Then $\hat{\CA}\,\Big((X, \gq)\Big)= \hat{R}$.
\end{lemma}

\begin{enumerate}[label=(e)]
\item  Let $\gp = (\varphi, \gr)$ where $\gr \in \Max\,(\mK[C])\setminus \{(C) \}$. By Theorem \ref{A23Feb15}.(5), $\CA/(\varphi, \gr)\simeq L_{\gr}[Y^{\pm 1}][K^{\pm 1};\tau]$ where $\tau(Y)= q^{-1}Y$ and so the automorphism $\tau$ is of (I)-type.
\end{enumerate}

\begin{lemma}\label{d15Mar15} 
Let $R= L_{\gr}[Y^{\pm 1}][K^{\pm 1};\tau]$ where $\tau(Y)=q^{-1}Y$. Then $\hat{\CA}\,\Big((\varphi, \gr)\Big) = \hat{R}.$
\end{lemma}

Let $\hat{\CA}$\,(fin. dim.) be the set of isomorphism classes of simple finite dimensional $\CA$-modules.
\begin{theorem}\label{A23Mar15} 
\begin{align*}
 \hat{\CA}\,{\rm (fin. dim.)} = \Big\{ \CA/(Y, E, \gp) \simeq \mK[K^{\pm 1}]/\gp \,|\, \gp \in \Max\,(\mK[K^{\pm 1}])  \Big\} = \bigsqcup_{P \in \CP}\hat{\CA}\,(P).
\end{align*}
\end{theorem}
\begin{proof}
A finite dimensional module over an infinite dimensional algebra has necessarily nonzero annihilator. Among the simple modules with nonzero annihilators, the modules described in (\ref{AAPU1}) are the only finite dimensional ones.
\end{proof}

For a module $M$, $\GK(M)$ is its \emph{Gelfand-Kirillov dimension}.
\begin{corollary}\label{a29Mar15} 
The Gelfand-Kirillov dimension of all simple, unfaithful, infinite dimensional $\CA$-modules is $1$.
\end{corollary}
\begin{proof}
The statement follows from the above classification of simple, unfaithful, infinite dimensional $\CA$-modules.
\end{proof}

\section{Classification of simple weight/$\mK[K^{\pm 1}]$-torsion $\CA$-modules} \label{SimWeight} 

Let $D=\mK[K, K^{-1}]$. The aim of this section is to give a classification of the set $\hat{\CA}$($D$-torsion) of isomorphism classes of simple $D$-torsion $\CA$-modules (Theorem \ref{B23Mar15} and Theorem \ref{25Mar15}). As a result, a classification of the set $\hat{\CA}$(weight) of isomorphism classes of simple weight $\CA$-modules is obtained by showing that $\hat{\CA}$($D$-torsion)$= \hat{\CA}$(weight) (Lemma \ref{a26Mar15}).

For the algebra $D=\mK[K,K^{-1}]$, let $D^{\circ}:= D \setminus \{ 0\}.$ It is obvious that the set $D^{\circ}$ is a left and right Ore set of the domain $\CA$. An $\CA$-module $M$ is called a $D$-\emph{torsion} $\CA$-module if $(D^{\circ})^{-1}M =0$. An important case of $D$-torsion $\CA$-modules are the weight $\CA$-modules.

Let $M$ be a weight $\CA$-module. For each $\gp \in \Supp(M)$, $M_{\gp}$ is the sum of all the $D$-submodules of $M$ that are isomorphic to the $D$-module $D/\gp$. If $\mK$ is an algebraically closed field then $\Max(D)=\{ (K-\mu)\,|\, \mu \in \mK^* \}$  and $M_{(K-\mu)}= \{ v \in M \,|\, Kv = \mu v \}$. The algebra $\CA$ admits the inner automorphism $\o_K: a \mapsto KaK^{-1}$. Each of the canonical generators $K^{\pm 1}, E, X$ and $Y$ of the algebra $\CA$ are eigenvectors of the automorphism $\o_K$ with eigenvalues, $1, q^2, q$ and $q^{-1}$, respectively.  So, the algebra $\CA= \bigoplus_{i \in \Z} \CA_i$ is a $\Z$-graded algebra where $\CA_i:= \{ a \in \CA \,|\, \o_K(a)= q^i a \} \neq 0$ and $\CA_0$ is the centralizer $C_{\CA}(K)= \{ a \in \CA\,|\, Ka=aK \}$ of the element $K$ in $\CA$.

Let $\sigma \in \Aut_{\mK}(D)$ where $\sigma(K)=qK$. For all $a_i \in \CA_i$, $Ka_i = \o_K(a_i)K= a_i q^iK = a_i \sigma^i(K)$. Therefore, for all $d \in D$, $da_i = a_i \sigma^i(d)$. If $M = \bigoplus M_{\gp}$ is a weight $\CA$-module then
\begin{align}
\CA_i M_{\gp}  \subseteq M_{\sigma^{-i}(\gp)}. \label{AiMp} 
\end{align}
So, each weight $\CA$-module $M$ is a direct sum of $\CA$-submodules
\begin{align}
M = \bigoplus_{\mathscr{O} \in \Max(D)/\langle \sigma \rangle} M(\mathscr{O}), \quad M(\mathscr{O}):= \bigoplus_{\gp \in \mathscr{O}} M_{\gp}, \label{AiMp1} 
\end{align}
where $\Max(D)/\langle \sigma \rangle$ is the set of $\langle \sigma \rangle$-orbits $\mathscr{O}$ in $\Max(D)$ (where $\mathscr{O}= \{\sigma^i(\gp)\,|\, i \in \Z \}$).

By (\ref{AiMp1}), if $M$ is a simple weight $\CA$-module then necessarily $\Supp(M) \subseteq \mathscr{O}$ for some orbit $\mathscr{O}$. In fact, for each simple infinite dimensional weight $\CA$-module $M$, $\Supp(M)$ is an orbit (Corollary \ref{b26Mar15}). For each $\gp \in \Max(D)$, the $\CA$-module
\begin{align*}
\CA(\gp):= \CA \otimes_D D/\gp = \bigoplus_{i \in \Z} \CA_i \otimes_D D/\gp
\end{align*}
is a weight $\CA$-module with $\Supp(\CA(\gp))= \mathscr{O}(\gp)= \{\sigma^i(\gp) \,|\, i \in \Z \}$ and $\CA(\gp)_{\sigma^{-i}(\gp)}= \CA_i \otimes_D D/\gp.$ Since $\CA_D$ is a free right $D$-module, $\CA(\gp) \simeq \CA/\CA \gp.$

\begin{lemma}\label{a26Mar15} 
$\hat{\CA}$\,{\rm (weight)} $= \hat{\CA}$\,{\rm ($D$-torsion)}.
\end{lemma}
\begin{proof}
Clearly, $\hat{\CA}$\,{\rm (weight)} $\subseteq \hat{\CA}$\,{\rm ($D$-torsion)}. The opposite inclusion follows from the fact that every simple $D$-torsion $\CA$-module $M$ is an epimorphic image of the weight $\CA$-module $\CA(\gp)$ for some $\gp \in \Max(D)$. Hence, $M$ is weight.
\end{proof}

Let
$$\Prim\,(\CA, D{\rm\text{-}torsion}):= \{ \ann_{\CA}(M)\,|\, M \in \hat{\CA}\,(D{\rm\text{-}torsion}) \}.$$
For each primitive ideal $\gp \in \Prim\,(\CA)$, let $\hat{\CA}\,(D{\rm\text{-}torsion}, \gp):= \{ M \in \hat{\CA}\,(D{\rm\text{-}torsion})\,|\, \ann_{\CA}(M)=\gp \}$. For each non-empty subset $S \subseteq \Prim\,(\CA)$, let
\begin{align*}
\hat{\CA}\,(D{\rm\text{-}torsion}, S):= \bigsqcup_{\gp \in S} \hat{\CA}\,(D{\rm\text{-}torsion}, \gp).
\end{align*}

\textbf{The set $\hat{\CA}\,(D{\rm\text{-}torsion},  \gp)$ where $\gp \in \Max\,(\CA)$.} The next theorem shows that $\Prim(\CA, D{\rm\text{-}torsion})= \Max\,(\CA) \,\, \sqcup \,\, \{0 \}$ and gives an explicit description of the sets $\hat{\CA}\,(D{\rm\text{-}torsion}, \gp)$ for all $\gp \in \Max\,(\CA)$.

\begin{theorem} \label{B23Mar15} 
Let $D=\mK[K,K^{-1}]$.
\begin{enumerate}
\item $\Prim\,(\CA, D{\rm\text{-}torsion}) = \Max(\CA) \, \sqcup \, \{0 \}$.  In particular, we have $\hat{\CA}\Big(D{\rm\text{-}torsion}, (Y)\Big) = \emptyset$ and $\hat{\CA}\Big(D{\rm\text{-}torsion}, (E)\Big) = \emptyset.$
\item Recall that $\Max(\CA) = \CP \,\sqcup \, \CQ \,\sqcup \, \CR.$ Then
\begin{enumerate}
\item $\hat{\CA}\,(D{\rm\text{-}torsion}, \CP)= \hat{\CA}\,({\rm fin.\,\,dim.})$. (See Theorem \ref{A23Mar15} for a classification.)
\item $\hat{\CA}\,(D{\rm\text{-}torsion}, \CQ)= \bigsqcup_{\gp \in \CQ} \hat{\CA}\,(D{\rm\text{-}torsion}, \gp)$. For each $\gp =(X, \gq) \in \CQ$ where $\gq \in \Max\,(\mK[Z])\setminus \{ (Z) \}$, $\CA/\gp \simeq L_\gq [K^{\pm 1}][Y^{\pm 1}; \rho]$ where $\rho$ is an $L_\gq$-automorphism of the polynomial algebra $D_\gq:= L_\gq[K^{\pm 1}]$ given by the rule $\rho(K)= qK$ and $\hat{\CA}\,(D{\rm\text{-}torsion}, \gp)= \widehat{\CA/\gp}\,(D_\gq{\rm\text{-}torsion}).$ The map
$$\Max(D_\gq)/\langle \rho \rangle \longrightarrow \widehat{\CA/\gp}\,(D_\gq{\rm\text{-}torsion}), \quad \mathscr{O}(\gm) \mapsto (\CA/\gp)/(\CA/\gp)\gm,$$
is a bijection with the inverse $M \mapsto \Supp_{D_{\gq}}(M).$
\item $\hat{\CA}\,(D{\rm\text{-}torsion}, \CR)= \bigsqcup_{\gp \in \CR} \hat{\CA}\,(D{\rm\text{-}torsion}, \gp)$. For each $\gp =(\varphi, \gr) \in \CR$ where $\gr \in \Max\,(\mK[C])\setminus \{ (C) \}$, $\CA/\gp \simeq L_\gr [K^{\pm 1}][Y^{\pm 1}; \rho]$ where $\rho$ is an $L_\gr$-automorphism of the polynomial algebra $D_\gr:= L_\gr[K^{\pm 1}]$ given by the rule $\rho(K)= qK$ and  $\hat{\CA}\,(D{\rm\text{-}torsion}, \gp)= \widehat{\CA/\gp}\,(D_\gr{\rm\text{-}torsion}).$ The map
$$\Max(D_\gr)/\langle \rho \rangle \longrightarrow \widehat{\CA/\gp}\,(D_{\gr}{\rm\text{-}torsion}), \quad \mathscr{O}(\gm) \mapsto (\CA/\gp)/(\CA/\gp)\gm,$$
is a bijection with the inverse $M \mapsto \Supp_{D_{\gr}}(M).$
\item The set $\hat{\CA}\,(D{\rm\text{-}torsion},0) \neq \emptyset$ is described in Theorem \ref{25Mar15}.
\end{enumerate}
\end{enumerate}
\end{theorem}
\begin{proof}
By Proposition \ref{a27Feb15}, $\Prim\,(\CA)= \Max(\CA) \, \sqcup \, \{ (Y),\, (E),\, 0 \}$ and $\Max(\CA)= \CP \, \sqcup \, \CQ \,\sqcup \,\CR$, by Corollary \ref{a24Feb15}.

(a) The statement (a) is obvious.

(b) Let $\gp =(X, \gq) \in \CQ$ where $\gq \in \Max\,(\mK[Z])\setminus \{ (Z) \}$. Since $L_\gq$ is a finite field extension of $\mK$, i.e., $\dim_{\mK}(L_\gq) < \infty$, we have the equality $\widehat{\CA/\gp}(D_{\gq}{\rm\text{-}torsion})= \hat{\CA}(D{\rm\text{-}torsion}, (X, \gq))$ and the statement (b) follows from Theorem \ref{A15Mar15}.(1).

(c) The statement (c) is treated almost identically to the statement (b) (with obvious modifications).

(i) $\hat{\CA}\Big(D{\rm\text{-}torsion}, (Y)\Big)=\emptyset:$ The algebra $\CA/(Y)$ is isomorphic to the algebra $U=D[E;\tau]$ where $\tau(K) = q^{-2}K$. So, $\hat{\CA}\Big(D{\rm\text{-}torsion}, (Y)\Big)= \hat{U}(D{\rm\text{-}torsion}, 0)$. For each $\gp \in \Max(D)$, the $U$-module $U(\gp):= U/(\gp) \simeq \bigoplus_{i \geqslant 0} E^i \otimes D/\gp$ contains the largest submodule $U(\gp)_+ := \bigoplus_{i \geqslant 1}E^i \otimes D/\gp$ and the factor module $U(\gp)/U(\gp)_+ \simeq D/\gp$ is simple with non-zero annihilator generated by the normal element $E$ of $U$, and the statement (i) follows.

(ii) $\hat{\CA}\Big(D{\rm\text{-}torsion}, (E)\Big)=\emptyset:$ By Theorem \ref{A23Feb15}.(3), $\CA/(E) \simeq D[Y;\tau]$ where $\tau(K)= qK$. The statement (ii) follows from the statement (i) by replacing $q^{-2}$ by $q$.

Now, statement 1 follows from the statements (i) and (ii).
\end{proof}

\noindent\textbf{The centralizer of $K$ in $\CA$.}
Now, we describe the centralizer $C_{\CA}(K)$ of the element $K$ in the algebra $\CA$. Recall that $C_{\CA}(K)= \{ a \in \CA \,|\, Ka = aK \}$. Note that $\CA$ is a $\Z$-graded algebra $\CA= \oplus_{i \in \Z} \CA_i$, where $\CA_i = \{ a \in \CA \,|\, KaK^{-1}= q^i a \}$. We see that $C_{\CA}(K)= \CA_0$.
\begin{proposition}\label{B23Feb15} 
Let $t:= YX$ and $u:= Y\varphi$, then $\CA_0 = \mK[K^{\pm 1}]\otimes \L$ is the tensor product of two algebras where $\L: = \mK \langle t, u \,|\, tu= q^2ut \rangle $ is a quantum plane, and $Z(\CA_0)= \mK[K^{\pm 1}]$.
\end{proposition}
\begin{proof}
Notice that the quantum plane $\mK_q[X, Y]$ is a $\Z$-graded algebra $\mK_q[X, Y]= \bigoplus_{j \in \Z}\mK[t]w_j$, where $w_0:=1$, for $j \geqslant 1$, $w_j= X^j$ and $w_{-j}= Y^j.$
Notice further that $U= \bigoplus_{i \in \N}\mK[K^{\pm 1}]E^i$. Since $\CA = U \otimes \mK_q[X, Y]$ (a tensor product of vector spaces),
\begin{align*}
\CA &= \bigoplus_{i \in \N}\mK[K^{\pm 1}]E^i \otimes \bigoplus_{j \in \Z}\mK[t]w_j= \bigoplus_{i\in \N,j\in\Z}\mK[K^{\pm 1}, t]E^iw_j \\
&= \bigoplus_{i,j \geqslant 0}\mK[K^{\pm 1}, t]E^iX^j \oplus \bigoplus_{i \geqslant 0, j\geqslant 1}\mK[K^{\pm 1}, t]E^i Y^j.
\end{align*}
Therefore,
$\CA_0 = \bigoplus_{i \geqslant 0} \mK[K^{\pm 1}, t]E^i Y^{2i}= \bigoplus_{i \geqslant 0}\mK[K^{\pm 1}, t](EY^2)^i.$
Then $\CA_0 = \mK \langle K^{\pm 1}, \,t, \, EY^2 \rangle$. Since $EY^2 = \frac{1}{q(1-q^2)}Y\varphi +\frac{q^3}{q^2-1} YX$, we see that $\CA_0 = \mK \langle K^{\pm 1}, \, t, \, Y\varphi \rangle = \mK[K^{\pm 1}] \otimes \L$ as required. Now, $Z(\CA_0)= \mK[K^{\pm 1}] \otimes Z(\L)= \mK[K^{\pm 1}] \otimes \mK = \mK[K^{\pm 1}]$.
\end{proof}

\noindent\textbf{The set $\hat{\CA}\,(D{\rm\text{-}torsion}, 0)$.}
Recall that $D=\mK[K^{\pm 1}]$, $\sigma \in \Aut_{\mK}(D)$ where $\sigma(K)= qK$ and $\CA_0 = D \otimes \L$ is  a tensor product of algebras. Then
\begin{align}
\CA_Y = \CA_0[Y^{\pm 1};\sigma] \label{AYA} 
\end{align}
where $\sigma(K)= qK, \sigma(t)= q^{-1}t$ and $\sigma(u)=qu$. Clearly,
\begin{align}
\CA_{Y, X, \varphi}= \CA_{Y, t, u} = (D \otimes \L_{t,u})[Y^{\pm 1}; \sigma] = {\CA}_{0,t, u}[Y^{\pm 1};\sigma] \label{AYA1} 
\end{align}
where $\CA_{0, t,u}= (\CA_0)_{t,u}$ is the localization of $\CA_0$ at the powers of the elements $t$ and $u$.
The inner automorphism $\o_Y: a \mapsto YaY^{-1}$ of the algebra $\CA_{Y, X, \varphi}$ preserves the subalgebras $\CA_0$ and ${\CA}_{0,t, u}$ since
\begin{align}
\o_Y(K)= qK, \quad \o_Y(t)= q^{-1}t \quad {\rm and} \quad \o_Y(u)=qu.
\end{align}
So, $\o_Y \in \Aut_{\mK}(\CA_0)$ and $\o_Y \in \Aut_{\mK}({\CA}_{0, t, u})$. Clearly, $(\o_Y)^i = \o_{Y^i}$ for all $i \in \Z$.

The group $\langle \sigma \rangle$ acts on the  maximal spectrum $\Max(D)$ of the algebra $D$. For each $\gp \in \Max(D)$, the orbit $\mathscr{O}(\gp):= \{ \sigma^i(\gp)\,|\, i\in \Z \}$ is infinite. The set of all $\langle \sigma \rangle$-orbits in $\Max(D)$ is denoted by $\Max(D)/\langle \sigma \rangle.$ \emph{For each orbit} $\mathscr{O}$,\emph{ we fix a representative} $\gp = \gp_{\mathscr{O}}$ \emph{of} $\mathscr{O}$. Then the factor algebra $\CA_0/(\gp) \simeq D/\gp \otimes \L = D/\gp \langle t, u \,|\, tu =q^2 ut \rangle$ is the quantum plane over the field $D/\gp$. Let $\widehat{D/\gp \otimes \L}(\infty{\rm\text{-}dim.})$ be the set of isomorphism classes of simple infinite dimensional $D/\gp \otimes \L$-modules. A classification of all simple $\L$-modules is given in \cite{Bav-SimpModQuanPlane}, \cite{Bav-OystGeneCrossPro} and \cite{Bav-OreSim-1999}. In particular, the map
\begin{align}
\widehat{D/\gp \otimes \L}(\infty{\rm\text{-}dim.}) \longrightarrow \widehat{D/\gp \otimes \L_{t,u}}\,(\soc_{D/\gp \otimes \L}\neq 0), \quad [N] \mapsto [N_{t,u}], \label{DpLM} 
\end{align}
is a bijection with the inverse $M \mapsto \soc_{D/\gp \otimes \L}(M).$

For each $[N] \in \widehat{D/\gp \otimes \L}(\infty{\rm\text{-}dim.})$, the vector space $\mathscr{N}(N):= \mK[Y, Y^{-1}] \otimes N = \bigoplus_{i \in \Z}Y^i \otimes N$ admits a structure of $\CA$-module given by the rule (where $v \in N$):
\begin{align}
K(Y^i \otimes v)&= q^{-i} Y^i \otimes Kv, & Y(Y^i \otimes v)&= Y^{i+1}\otimes v, \notag \\
X(Y^i \otimes v)&= q^iY^{i-1}\otimes t v, & E(Y^i \otimes v)&= Y^{i-2} \otimes \frac{q^{-i}u - q^it}{q^{-1}-q}v. \label{ModN1} 
\end{align}
One can verify this directly using the defining relations of the algebra $\CA$ or, alternatively, this will be proved in the proof of Theorem \ref{25Mar15}. This construction gives all the elements of the set $\hat{\CA}(D{\rm\text{-}torsion}, 0)$ (Theorem \ref{25Mar15}).

\begin{lemma} \label{a5Jun15} 
Let $M$ be a simple weight $\CA$-module. If $M_Y =0$ then $XM=0$, and so $\ann_{\CA}(M) \neq 0$.
\end{lemma}
\begin{proof}
Suppose that $M_Y =0$. Then $N:= \{ m \in M \,|\, Ym=0 \}$ is a nonzero $\mK[K^{\pm 1}]$-submodule of $M$. Hence, there exists a nonzero weight vector, say $v$, such that $Yv=0$. Then $\CA v = M$, by the simplicity of the $\CA$-module $M$. Then $\CN:= \CA Xv$ is a submodule of $M$ such that $\CN \neq M$. Therefore, $\CN=0$ (since $M$ is a simple $\CA$-module), and the lemma follows.
\end{proof}

 Let $R$ be a ring and $M$ be a left $R$-module. Let $\sigma$ be an automorphism  of $R$. The \emph{twisted module} $M^{\sigma}$ is a module obtained from $M$ by twisting the action of $R$ on $M$ by $\sigma$, i.e., $r \cdot m = \sigma(r)m$ for $r \in R$ and $m \in M.$ The next theorem describes the set $\hat{\CA}\,$($D$-torsion,\,0).

\begin{theorem}\label{25Mar15} 
Let $D= \mK[K, K^{-1}]$ and $\sigma \in \Aut_{\mK}(D)$ where $\sigma(K)=qK$. Then
\begin{enumerate}
\item $\hat{\CA}\,(D{\rm\text{-}torsion}, 0) = \bigsqcup\limits_{\mathscr{O} \in \Max(D)/\langle \sigma \rangle} \hat{\CA}\,(D{\rm\text{-}torsion}, 0, \mathscr{O})$ where $\hat{\CA}\,(D{\rm\text{-}torsion}, 0, \mathscr{O}):= \{ [M] \in \hat{\CA}\,(D{\rm\text{-}torsion}, 0)\,|\, \Supp(M)= \mathscr{O} \}$.
\item For each orbit $\mathscr{O}= \mathscr{O}(\gp)\in \Max(D)/\langle \sigma \rangle$ where $\gp = \gp_{\mathscr{O}}$ is a fixed element of $\mathscr{O}$, the map
\begin{align*}
\widehat{D/\gp \otimes \L}(\infty{\rm\text{-}dim.}) \longrightarrow \hat{\CA}\,(D{\rm\text{-}torsion}, 0, \mathscr{O}), \quad [N] \mapsto [\mathscr{N}(N)],
\end{align*}
is a bijection with the inverse $\mathscr{N} \mapsto \mathscr{N}_\gp.$
\item For all $[M] \in \hat{\CA}\,(D{\rm\text{-}torsion},0)$, $\GK(M)=2$.
\end{enumerate}
\end{theorem}
\begin{proof}
1 and 2. The algebra $\CA$ is a subalgebra of the simple algebra $\CA_{Y, X, \varphi}= \CA_{Y, t, u} = (D \otimes \L_{t,u})[Y^{\pm 1}; \sigma]$ (Lemma \ref{a22Feb15}.(2)). Let $[M] \in \hat{\CA}$\,($D$-torsion,\,0). Since  the elements $X$ and $\varphi$ are normal in $\CA$, $\tor_{X, \varphi}(M)=0$ (otherwise, either $XM=0$ or $\varphi M=0$, a contradiction). Now, by Lemma \ref{a5Jun15}, the $\CA$-module $M$ is $(Y, t, u)$-torsionfree. So, the map
\begin{align}
\hat{\CA}\,(D{\rm\text{-}torsion}, 0) \rightarrow \widehat{\CA_{Y, t, u}}\,(D{\rm\text{-}torsion, \soc_\CA \neq 0}), \quad [M] \mapsto [M_{Y,t, u}],
\end{align}
is a bijection with the inverse $[N] \mapsto [\soc_{\CA}(N)]$ where the condition $\soc_{\CA} \neq 0$ means that an $\CA_{Y, t, u}$-module has non-zero socle as an $\CA$-module. 
Since $q$ is not a root of 1, the centralizer of the element $K$ in $\CA_{Y, t, u}=(D \otimes \L_{t,u})[Y^{\pm 1};\sigma]$ is $D \otimes \L_{t,u}$. Therefore, every module $V \in \widehat{\CA_{Y, t, u}}\,(D{\rm\text{-}torsion}, \soc_{\CA}\neq 0)$ is of the type
$$V= \CA_{Y, t, u} \otimes_{D \otimes \L_{t,u}} N' = \mK[Y^{\pm 1}] \otimes N' = \bigoplus_{i \in \Z}Y^i \otimes N'$$
for some $N' \in \widehat{D/\gp \otimes \L_{t,u}}$ where $\gp \in \Max(D)$. In more detail, the $D$-module $N'$ is a direct sum of copies of $D/\gp$. For each $i \in \Z$, the left $D$-module $Y^i \otimes N'$ is a direct sum of copies $D/\sigma^i(\gp)$  (since $\sigma^i(\gp) (Y^i \otimes N')= Y^i \gp \otimes N' = Y^i \otimes \gp N'=0$). The ideals $\{ \sigma^i(\gp)\,| \, i \in \Z \}$ are distinct (as $q$ is not a root of 1). Therefore, the $D$-modules $Y^i \otimes N'$ and $Y^j \otimes N'$ are not isomorphic for all $i \neq j$. Any nonzero submodule $M'$ of the $\CA$-module $\CA_{Y, t,u} \otimes_{D \otimes \L_{t,u}} N' = \bigoplus_{i \in \Z} Y^i \otimes N'$ is necessarily of the form $\bigoplus_{i \in \CI}L_i$ for some $\CI \subseteq \Z$ where for each $i \in \CI$, $L_i$ is a nonzero $\CA_{0, t,u}$-submodule of the simple $\CA_{0, t,u}$-module $Y^i \otimes N'$. Therefore, $L_i = Y^i \otimes N'$ for all $i \in \CI$. Then necessarily $I = \Z$, i.e., $M' = \CA_{Y, t, u}\otimes_{D \otimes \L_{t,u}} N'$.

The action of the elements $K, E, X$ and $Y$ are given by (\ref{ModN1}). In more detail, notice that
\begin{align}
EY = \frac{q^{-1}\varphi - qX}{q^{-1}-q} \label{ADAS2} 
\end{align}
since $EY = \varphi + qYE = \varphi + q^2(EY-X)$ and the equality (\ref{ADAS2}) follows. Then
\begin{align*}
E(Y^i \otimes v) & = EY \cdot Y^{i-1} \otimes v = \frac{q^{-1}\varphi - qX}{q^{-1}-q} Y^{i-1} \otimes v \\
&= Y^{i-1}\frac{q^{-i}\varphi - q^i X}{q^{-1}-q} \otimes v = Y^{i-2} \otimes \frac{q^{-i}u -q^i t}{q^{-1}-q}v.
\end{align*}
Notice that $\Supp(V) = \mathscr{O}(\gp)$ as $V_{\sigma^i(\gp)}= Y^i\otimes N' \neq 0$ for all $i \in \Z$. By (\ref{DpLM}), we may assume that $N'= N_{t,u}$ where $N \in \widehat{D/\gp \otimes \L}(\infty{\rm\text{-}dim.})$. Then by (\ref{ModN1}), the direct sum $\mathscr{N}(N)= \bigoplus_{i \in \Z} Y^i \otimes N$  is an $\CA$-submodule of $V$. The $\CA_0$-module $N$ is simple, hence so are the $\CA_0$-modules $Y^i \otimes N \simeq N^{\o_{Y^i}^{-1}}$. Therefore, $\soc_{\CA}(V) = \mathscr{N}(N)$. Since $\Supp(\soc_{\CA}(V))= \mathscr{O}(\gp)$ as $\soc_{\CA}(V)_{\sigma^i(\gp)}= Y^i \otimes N \neq 0$ for all $i \in \Z$, statement 1 follows.

Clearly, $V= \CA_{Y,t,u}\otimes_{\CA_{0,t,u}}N' = \CA_{Y, t,u} \otimes_{\CA_{0,t,u}}N_{t,u}= \CA_{Y,t,u}\otimes_{\CA_{0,t,u}}\CA_{0,t,u} \otimes_{\CA_0} N \simeq \CA_{Y, t,u} \otimes_{\CA_0}N$. Two simple $D$-torsion modules $V \simeq \CA_{Y, t,u} \otimes_{\CA_0}N$ and $U \simeq \CA_{Y, t,u} \otimes_{\CA_0}L$ are isomorphic where $[N] \in \widehat{D/\gp \otimes \L}$ and $[L] \in \widehat{D/\gq \otimes \L}$ iff  $\mathscr{O}:= \Supp(V) = \Supp(U)$ and $_{{\CA}_{0,t,u}}V_\gp \simeq \, _{{\CA}_{0,t,u}}U_\gp$ where $\gp = \gp_{\mathscr{O}}$ is a fixed element of the orbit $\mathscr{O}$ and $\CA_{0,t,u}= (\CA_0)_{t,u}$.

Since $V_{\gp}= N_{t,u}$ and $U_\gp = U_{\sigma^i(\gq)}= Y^i \otimes L_{t,u} \simeq L_{t,u}^{\o_{Y^{-i}}}$ where $\gp = \sigma^i(\gq)$ for some $i \in \Z$, the simple ${\CA}_{0,t,u}$-modules $V_{\gp}$ and $U_{\gp}$ are isomorphic iff the simple $\CA_0$-modules $\soc_{\CA_0}(V_{\gp})= N$ and $\soc_{\CA_0}(U_{\gp}) = \soc_{\CA_0}(L_{t,u}^{\o_{Y^{-i}}}) \simeq L^{\o_{Y^{-i}}}$ are isomorphic. Now, statement 2 follows.

3. Statement 3 follows from statement 2 and the fact that $\GK(N)=1$ for all $[N] \in \widehat{D/\gp \otimes \L}(\infty{\rm\text{-}dim.})$.
\end{proof}
So, Theorem \ref{B23Mar15} and Theorem \ref{25Mar15} give an explicit description of the set $\hat{\CA}$\,(weight) $= \hat{\CA}$\,($D$-torsion) (see Lemma \ref{a26Mar15}).

\begin{corollary}\label{b26Mar15} 
Let $[M] \in \hat{\CA}\,({\rm weight}, \infty{\rm\text{-}dim.})$
\begin{enumerate}
\item $\Supp(M)$ is a single orbit in $\Max(D)/\langle \sigma \rangle.$
\item $\ann_\CA(M) \neq 0$ iff $\dim_{\mK}(M_\gp) < \infty$ for all $\gp \in \Supp(M)$. In this case, the dimension $\dim_{\mK}(M_\gp)$ is the same for all $\gp \in \Supp(M)$; and if, in additional, $\mK$ is an algebraically closed field. Then $\dim_{\mK}(M_{\gp})=1$ for all $\gp \in \Supp(M)$.
\end{enumerate}
\end{corollary}
\begin{proof}
The result follows at once from the classification of simple weight $\CA$-modules.
\end{proof}

\small{

\begin{minipage}[t]{0.5\textwidth}
V. V. Bavula 

Department of Pure Mathematics

University of Sheffield

Hicks Building

Sheffield S3 7RH

UK

email: v.bavula@sheffield.ac.uk 
\end{minipage}
\begin{minipage}[t]{0.5\textwidth}
T. Lu

Department of Pure Mathematics

University of Sheffield

Hicks Building

Sheffield S3 7RH

UK

email: smp12tl@sheffield.ac.uk
\end{minipage}

\end{document}